\documentclass[12pt,reqno,a4]{amsart}

\usepackage[latin1]{inputenc}
\usepackage{amsmath}
\usepackage{amsthm}
\usepackage{amssymb}
\usepackage{amsfonts}
\usepackage{enumerate}
\usepackage{stmaryrd}
\usepackage{enumitem}
\usepackage{centernot}
\usepackage{mathrsfs}
\usepackage{listings}
\usepackage{multicol}
\usepackage[noend]{algorithmic} % for the algorithmic enviroment

\usepackage{array}
\newcolumntype{L}[1]{>{\raggedright\let\newline\\\arraybackslash\hspace{0pt}}m{#1}}
\newcolumntype{C}[1]{>{\centering\let\newline\\\arraybackslash\hspace{0pt}}m{#1}}
\newcolumntype{R}[1]{>{\raggedleft\let\newline\\\arraybackslash\hspace{0pt}}m{#1}}
\usepackage{wrapfig}
\usepackage[colorlinks,plainpages,hypertexnames=false,plainpages=false]{hyperref}
\hypersetup{urlcolor=blue, citecolor=blue, linkcolor=blue}
\usepackage{tikz}
\usetikzlibrary{matrix}
\usetikzlibrary{calc}
\usetikzlibrary{patterns}
\pgfdeclarepatternformonly{senwStripes}{\pgfpoint{0cm}{0cm}}{\pgfpoint{1cm}{1cm}}{\pgfpoint{1cm}{1cm}}
{
  \foreach \i in {0.0, 0.2,...,0.8}
  {
    \pgfpathmoveto{\pgfpoint{\i cm}{0cm}}
    \pgfpathlineto{\pgfpoint{0cm}{\i cm}}
    \pgfpathlineto{\pgfpoint{0cm}{\i cm + 0.1cm}}
    \pgfpathlineto{\pgfpoint{\i cm + 0.1cm}{0cm}}
    \pgfpathclose%
    \pgfusepath{fill}
    \pgfpathmoveto{\pgfpoint{1cm}{\i cm}}
    \pgfpathlineto{\pgfpoint{\i cm}{1cm}}
    \pgfpathlineto{\pgfpoint{\i cm + 0.1cm}{1cm}}
    \pgfpathlineto{\pgfpoint{1cm}{\i cm + 0.1cm}}
    \pgfpathclose%
    \pgfusepath{fill}
  }
}
\newtheoremstyle{theoremstyle}
  {10pt}      %  Space above
  {5pt}       %  Space below
  {\itshape}  %  Body font
  {}          %  Indent amount (empty = no indent, \parindent = para indent)
  {\bfseries} %  Thm head font
  {}         %  Punctuation after thm head
  {\newline}      %  Space after thm head: " " = normal interword space;
  {}          %  Thm head spec (can be left empty, meaning `normal')

\newtheoremstyle{algorithmstyle}
  {10pt}      %  Space above
  {5pt}       %  Space below
  {}  %  Body font
  {}          %  Indent amount (empty = no indent, \parindent = para indent)
  {\bfseries} %  Thm head font
  {}         %  Punctuation after thm head
  { }      %  Space after thm head: " " = normal interword space;
  {}          %  Thm head spec (can be left empty, meaning `normal')

\newtheoremstyle{examplestyle}
  {10pt}      %  Space above
  {5pt}       %  Space below
  {}          %  Body font
  {}          %  Indent amount (empty = no indent, \parindent = para indent)
  {\bfseries} %  Thm head font
  {}         %  Punctuation after thm head
  {\newline}      %  Space after thm head: " " = normal interword space;
  {}          %  Thm head spec (can be left empty, meaning `normal')
\theoremstyle{theoremstyle}
\newtheorem{theorem}{Theorem}[section]
\newtheorem{lemma}[theorem]{Lemma}
\newtheorem{proposition}[theorem]{Proposition}
\newtheorem{corollary}[theorem]{Corollary}
\theoremstyle{examplestyle}
\newtheorem{example}[theorem]{Example}
\newtheorem{definition}[theorem]{Definition}
\newtheorem{remark}[theorem]{Remark}
\newtheorem{convention}[theorem]{Convention}
\theoremstyle{algorithmstyle}
\newtheorem{algorithm}[theorem]{Algorithm}
\newcommand{\NN }{\mathbb{N}}
\newcommand{\CC }{\mathbb{C}}
\newcommand{\RR }{\mathbb{R}}
\newcommand{\QQ }{\mathbb{Q}}
\newcommand{\ZZ }{\mathbb{Z}}
\newcommand{\FF }{\mathbb{F}}
\newcommand{\ov}{\overline}
\newcommand{\Rt}{{R\llbracket t\rrbracket}}
\newcommand{\Rtx}{{R\llbracket t\rrbracket [x]}}
\newcommand{\suchthat}{\;\ifnum\currentgrouptype=16 \middle\fi|\;}
\newcommand{\bigmid}{\left.\vphantom{\Big\{} \suchthat \vphantom{\Big\}}\right.}
\newcommand{\pinitial}[2]{\overline{\initial_{#1}({#2})|}_{t=1}}

\DeclareMathOperator{\val}{val}
\DeclareMathOperator{\Mon}{Mon}
\DeclareMathOperator{\lm}{LM}
\DeclareMathOperator{\lt}{LT}

\DeclareMathOperator{\initial}{in}
\DeclareMathOperator{\Relint}{Relint}
\DeclareMathOperator{\Cone}{Cone}
\DeclareMathOperator{\Conv}{Conv}
\DeclareMathOperator{\Flip}{Flip}
\DeclareMathOperator{\Trop}{\mathcal T}
\DeclareMathOperator{\relint}{relint}
\DeclareMathOperator{\Grass}{Grass}
\DeclareMathOperator{\Lin}{Lin}

\newcommand{\lang}[1]{}
\begin{document}

   \parindent0cm

   \title[Computing tropical varieties]{Computing tropical varieties\\over fields with valuation}
   \author{Thomas Markwig}
   \address{Eberhard Karls Universit\"at T\"ubingen\\
     Fachbereich Mathematik\\
     Auf der Morgenstelle 10\\
     72076 Tübingen, Germany
     }
   \email{keilen@math.uni-tuebingen.de}
   \urladdr{https://www.math.uni-tuebingen.de/\textasciitilde keilen}

   \author{Yue Ren}
   \address{Ben-Gurion University of the Negev\\
     Department of Mathematics\\
     P.O.B. 653\\
     8410501 Be'er Sheva, Israel
     }
   \email{reny@post.bgu.ac.il}
   \urladdr{https://www.math.bgu.ac.il/\textasciitilde reny}

   \thanks{The author was supported by ...}


   \date{December, 2016}


   \begin{abstract}
     We show how tropical varieties of ideals $I\unlhd K[x]$ over a
     field $K$ with non-trivial valuation can always be traced back to
     tropical varieties of ideals $\pi^{-1}I\unlhd\Rtx$ over some
     dense subring $R$ in its ring of integers. Moreover, for
     homogeneous ideals, we present algorithms on how the latter can
     be computed in finite time, provided that $\pi^{-1}I$ is
     generated by elements in $R[t,x]$. While doing so, we also comment
     on the computation of the Gr\"obner polytope structure and $p$-adic
     Gr\"obner bases using our framework. %All algorithms solely rely on
     %existing standard basis techniques.
   \end{abstract}

   \maketitle

   %%%%%%%%%%%%%%%%%%%%%%%%%%%%%%%%%%%%%%%%%%%%%%%%%%%%%%%%%%%%%%%%%%%%%%%%%%%%%%%%%%
   \section{Introduction}

   Tropical varieties are commonly described as combinatorial shadows of their algebraic counterparts, and computing tropical varieties is an algorithmically highly challenging task, requiring sophisticated techniques from computer algebra and convex geometry.

   The first techniques were developed by Bogart, Jensen, Speyer, Surmfels and Thomas \cite{BJSST07}, who focused on homogeneous ideals over $\CC$ with the trivial valuation, which allowed them to rely on classical Gr\"obner basis methods. Furthermore, the authors showed that, under sensible conditions, their techniques can be used over the field of Puiseux series $\mathbb C\{\!\{t\}\!\}$ with its natural valuation, by regarding $t$ as a variable in the polynomial ring instead of a uniformizing parameter in the coefficient ring. The inhomogeneity of the resulting ideal in $\CC[t,x]$ can be worked around through homogenization and dehomogenization.
   In order to adapt these techniques to the field of $p$-adic numbers and the $p$-adic valuation, Chan and Maclagan adapted the classical theory of Gr\"obner bases \cite{CM13} to take the valuation on the ground field into account, instead of solely relying on monomial orderings. % Moreover, they also showed that their adapted Gr\"obner bases satisfies the same classical bounds and that there exist families of ideals for which their adapted Gr\"obner bases are smaller than their classical pendants.

   In this article, in Section~\ref{sec:tropicalVariety}, we discuss another approach to compute tropical varieties over an arbitrary field with valuation, which can be regarded as a generalisation of the trick used for $\mathbb C\{\!\{t\}\!\}$.
   For that, we combine the existing notions of tropical varieties over power series \cite{Touda05,BT07,PS13} with the concept of tropical varieties over coefficient rings \cite[Section 1.6]{MS15}.
   Compared to \cite{CM13}, the approach relies on existing standard basis theory, which not only allows us to exploit the highly optimized implementations that exist in many established computer algebra systems such as \textsc{Singular} \cite{DGPS16} or \textsc{Macaulay2} \cite{M2}, it also gives us access to a highly active field of research.

   Moreover, in Section~\ref{sec:groebnerComplex}, we improve on the techniques in \cite{BJSST07} by avoiding homogenization and dehomogenization. We also touch upon the topic on how to compute $p$-adic Gr\"obner bases in our framework.
   In Section~\ref{sec:computation} we present the algorithms for
   computing tropical varieties and in Section~\ref{sec:optimization}
   we touch upon possible optimizations that are exclusive to
   non-trivial valuations.

   All algorithms in this article are
   implemented in the \textsc{Singular} library \texttt{tropical.lib}
   \cite{tropicallib}, relying on the \textsc{gfanlib} interface
   \texttt{gfan.lib} \cite{gfan,gfanlib} for computations in convex
   geometry. They are publicly available as part of the official
   \textsc{Singular} distribution.

   \renewcommand{\emph}[1]{\textit{\textcolor{red}{#1}}}
   \section{Tracing tropical varieties to a trivial valuation}\label{sec:tropicalVariety}

   The aim of this section is to show how tropical varieties over valued fields can be traced back to tropical varieties over integral power series. The linchpin of the section is to show that initial ideals over valued fields can be described through initial ideals over integral power series, the remaining results then follow naturally from this. To fix the notation, we will begin by recalling some very basic notions in tropical geometry that are of immediate relevance to us.

   \begin{convention}\label{con:coefficients}
     For the remainder of the article, fix a complete field $K$ with non-trivial discrete valuation $\nu:K\rightarrow\RR\cup\{\infty\}$ and a uniformizing parameter $p\in K$. Let $\mathcal O_K$ be its ring of integers and let $\mathfrak K$ denote its residue field.
     Let $R\leq \mathcal O_K$ be a dense, noetherian subring. By Cohen's Structure Theorem, we have two exact sequences
     \begin{center}
       \begin{tikzpicture}[descr/.style={fill=white,inner sep=2pt}]
         \matrix (m) [matrix of math nodes, row sep=2em,
         column sep=2.5em, text height=1.5ex, text depth=0.25ex]
         { 0 & \phantom{{}_{\langle p-t \rangle}}\langle p-t \rangle\cdot R\llbracket t\rrbracket_{\langle p-t \rangle}
           & \phantom{{}_{\langle p-t \rangle}}R\llbracket t\rrbracket_{\langle p-t \rangle}
           & K & 0, \\
           0 & \langle p-t \rangle\cdot R\llbracket t\rrbracket & R\llbracket t\rrbracket
           & \mathcal O_K & 0. \\ };
         \path[->,font=\scriptsize,shorten >=-1.9em]
         (m-1-1) edge (m-1-2)
         (m-1-2) edge (m-1-3);
         \path[->,font=\scriptsize]
         (m-1-3) edge (m-1-4)
         (m-1-4) edge (m-1-5)
         (m-2-1) edge (m-2-2)
         (m-2-2) edge (m-2-3)
         (m-2-3) edge node[below] {$t\longmapsto p$} node[above] {$\pi$} (m-2-4)
         (m-2-4) edge (m-2-5)
         (m-2-2) edge (m-1-2)
         (m-2-3) edge (m-1-3)
         (m-2-4) edge (m-1-4); %node[auto] {$\iota$} (m-1-4);
       \end{tikzpicture}
     \end{center}
     Moreover, fix a multivariate polynomial ring $K[x]=K[x_1,\ldots,x_n]$. By abuse of notation, we will also use
     $\pi$ to refer to both the map $\Rtx\rightarrow \mathcal O_K[x]$ as well as
     the composition $\Rtx\rightarrow \mathcal O_K[x] \hookrightarrow K[x]$.
   \end{convention}

    \begin{example}[$p$-adic numbers]
      The most important example is the field $K:=\QQ_p$ of $p$-adic
      numbers with $\mathcal O_K:=\ZZ_p$ the
      ring of $p$-adic integers. Then $R:=\ZZ\leq\ZZ_p$ is a natural dense
      subring, which is computationally easy to work over. The exact
      sequences in Convention~\ref{con:coefficients} merely reflect
      the presentation of $p$-adic integers as power series in $p$:
      \begin{center}
        \begin{tikzpicture}[descr/.style={fill=white,inner sep=2pt}]
          \matrix (m) [matrix of math nodes, row sep=1.5em,
          column sep=2.5em, text height=1.5ex, text depth=0.25ex]
          { 0 & \langle p-t \rangle\cdot \ZZ\llbracket t\rrbracket_{\langle p-t \rangle} [x]
            & \ZZ\llbracket t\rrbracket_{\langle p-t \rangle} [x]
            & \QQ_p[x] & 0, \\
            0 & \langle p-t \rangle\cdot \ZZ\llbracket t\rrbracket[x] & \ZZ\llbracket t\rrbracket[x] & \ZZ_p[x] & 0. \\ };
          \path[->,font=\scriptsize]
          (m-1-1) edge (m-1-2)
          (m-1-2) edge (m-1-3)
          (m-1-3) edge (m-1-4)
          (m-1-4) edge (m-1-5)
          (m-2-1) edge (m-2-2)
          (m-2-2) edge (m-2-3)
          (m-2-3) edge node[below] {$t\longmapsto p$} node[above] {$\pi$} (m-2-4)
          (m-2-4) edge (m-2-5)
          (m-2-2) edge (m-1-2)
          (m-2-3) edge (m-1-3)
          (m-2-4) edge (m-1-4); %node[auto] {$\iota$} (m-1-4);
       \end{tikzpicture}
     \end{center}
   \end{example}

     \begin{example}
       Given the choice of $R\leq \mathcal O_K$ in Convention \ref{con:coefficients}, choosing $R:=\mathcal O_K$ is always possible.
       However, in many examples there are natural choices for $R$, which are computationally much easier to handle than $\mathcal O_K$ itself:
       \begin{enumerate}[leftmargin=*]
       \item $K=k(\!(t)\!)$ the field of Laurent series over a field
         $k$ with $\mathcal O_K=k\llbracket t \rrbracket$ the ring of
         power series over $k$, $R=k[t]$ and $p=t$; e.g.~$k=\FF_q$ with $q$ a prime power, as used in \cite[Section 7]{SpeyerSturmfels04} or \cite{Kalinin13},
         or $k=\QQ$ as considered in \cite{BJSST07}, see Example~\ref{ex:puiseuxSeries}.
       \item Finite extensions $K$ of $\QQ_p$ and $\FF_q(\!(t)\!)$,
         i.e. all local fields with non-trivial valuation, and also
         all higher dimensional local fields.
       \item $\mathcal O_K$ any completion of a localization of a Dedekind domain
         $R$ at a prime ideal $P\unlhd R$, $p\in P$ a suitable
         element. Note that $p$ does not need to generate $P$ and
         hence $\mathcal{O}_K$ need not be the completion with respect to the
         ideal generated by $p$,
         e.g. $R=\ZZ[\sqrt{-5}]$, $P=\langle 2,1+\sqrt{-5}\rangle$ and
         $p=2$.
       \item For an odd choice of $R$, consider $K:=\QQ(s)(\!(t)\!)$ so that $\mathcal O_K=\QQ(s)\llbracket t\rrbracket$. Set $R:=S^{-1}\QQ[s,t]$, where $S:=\QQ[s,t]\setminus (\langle
         t-1,s\rangle\cup\langle x\rangle)$ is multiplicatively closed as the complement of two prime ideals. Then $R$ is a non-catenarian, dense subring of $\mathcal O_K$.
       \end{enumerate}
     \end{example}

   \begin{definition}[initial forms, initial ideals, tropical varieties over valued fields]\label{def:valuedInitial}
     For a polynomial $f=\sum_{\alpha\in\NN^n} c_\alpha\cdot
     x^\alpha\in K[x]$ and a weight vector $w\in\RR^n$, we define
     the \emph{initial form} of $f$ with respect to $w$ to be:
     \begin{displaymath}
       \initial_{\nu,w}(f) := \sum_{\substack{w\cdot \alpha-\nu(c_\alpha) \\ \text{maximal}}}
       \ov{c_\alpha \cdot p^{-\nu(c_\alpha)}}\cdot x^\alpha \in \mathfrak{K}[x].
     \end{displaymath}
     For any subset $I\subseteq K[x]$ and a weight vector
     $w\in\RR^n$, we define the \emph{initial ideal} of I with
     respect to $w$ to be:
     \begin{displaymath}
       \initial_{\nu,w}(I) := \langle \initial_{\nu,w}(f)\mid f\in I\rangle \unlhd \mathfrak{K}[x].
     \end{displaymath}
     We refer to the set of weight vectors for which the initial ideal contains no monomial as the \emph{tropical variety} of $I$,
     \begin{displaymath}
       \Trop_{\hspace{-0.1cm}\nu}(I) := \left\{w\in\RR^n \suchthat \initial_{\nu,w}(I) \text{ monomial free}\right\}.
     \end{displaymath}
   \end{definition}

   \begin{theorem}[{\cite[Theorem 3.3.5]{MS15}}]
     Let $I\unlhd K[x]$ define an irreducible subvariety in $(K^\ast)^n$ of dimension $d$. Then $\Trop_{\hspace{-0.1cm}\nu}(I)$ is the support of a pure polyhedral complex of same dimension that is connected in codimension $1$.
   \end{theorem}

   Next, we will introduce tropical varieties in $\Rtx$, and show how
   a certain class of them relates to tropical varieties in $K[x]$. In
   particular, we will note that those tropical varieties in $\Rtx$ are
   pure and connected in codimension $1$. We begin by introducing
   initial forms and initial ideals in $\Rtx$ and show how they can be
   used to describe their pendants in $K[x]$.

   \begin{definition}[initial forms, initial ideals]\label{def:nonValuedInitial}
     Given an element $f=\sum_{\beta,\alpha} c_{\alpha,\beta}\cdot
     t^\beta x^\alpha\unlhd \Rtx$ and a weight vector
     $w\in\RR_{<0}\times\RR^n$, we define the \emph{initial form} of $f$
     with respect to $w$ to be
     \begin{displaymath}
       \initial_{w}(f) := \sum_{w\cdot (\beta,\alpha) \text{ maximal}}
       c_\alpha t^\beta x^\alpha \in R[t,x].
     \end{displaymath}
     Given an ideal $I\unlhd \Rtx$ and a weight vector
     $w\in\RR_{<0}\times\RR^n$, we define the \emph{initial ideal} of
     $I$ with respect to $w$ to be:
     \begin{displaymath}
       \initial_{w}(J) := \langle \initial_{w}(f)\mid f\in J\rangle \unlhd R[t,x].
     \end{displaymath}
     This can be thought of as a natural extension of Definition \ref{def:valuedInitial} with trivial valuation on the coefficients.
     Note that we only allow weight vectors with negative weight in $t$, so that our result lies in a polynomial ring.
   \end{definition}

    \begin{example}[$p$-adic numbers]\label{ex:Chan}
     Let us consider the example in \cite[Chapter 3.6]{Chan13}, the ideal
     \begin{displaymath}
       I=\langle 2x_1^2+3x_1x_2+24x_3x_4, 8x_1^3+x_2x_3x_4+18x_3^2x_4\rangle \unlhd\QQ_3[x_1,\ldots,x_4]
     \end{displaymath}
     over the $3$-adic number $Q_3$, so that
     \begin{displaymath}
       \pi^{-1}I=\langle 3-t,2x_1^2+3x_1x_2+24x_3x_4, 8x_1^3+x_2x_3x_4+18x_3^2x_4\rangle\unlhd\ZZ\llbracket t\rrbracket[x],
     \end{displaymath}
     and the weight vector $(-1,w)\in\RR_{<0}\times\RR^4$, $w:=(1,11,3,19)$.
     A short computation yields
     \begin{displaymath}
       \initial_{(-1,w)}(\pi^{-1}I) = \langle 3, x_1^2, tx_1x_3x_4,t^3x_1x_2^2x_3, t^4x_1x_2^4, t^3x_3^4x_4^2 \rangle,
     \end{displaymath}
     and the similarity to the initial ideal of $I$ under the $3$-adic valuation is no coincidence:
     \begin{displaymath}
       \initial_{\nu_3,w}(I) = \langle x_1^2, x_1x_3x_4,x_1x_2^2x_3, x_1x_2^4, x_3^4x_4^2 \rangle \unlhd\FF_3[x].
     \end{displaymath}
   \end{example}

   \begin{proposition}\label{prop:initial}
     For any ideal $I\unlhd \mathcal O_K[x]$ and any weight vector $w\in\RR^n$, we have:
     \begin{displaymath}
       \pinitial{(-1,w)}{\pi^{-1}I}=\initial_{\nu,w}(I),
     \end{displaymath}
     where $\ov{(\cdot)}$ denotes the canonical projection $\ov{(\cdot)}:R[x]\rightarrow \mathfrak{K}[x]$.
   \end{proposition}
   \begin{proof}\begin{description}[leftmargin=0.5em,font=\normalfont]
     \item[$\supseteq$]
       Any term $s\in \mathcal O_K[x]$ is of the form $s=(\sum_\beta c_\beta p^\beta)\cdot  x^\alpha$ with $p\nmid c_\beta$ for all $\beta\in\NN$.
       Then the element $s':=(\sum_\beta c_\beta t^\beta)\cdot  x^\alpha\in \Rtx$
       is a natural preimage of it under $\pi$ for which we have
       \begin{displaymath}
         \initial_{\nu,w}(s)=\ov c_{\beta_0} \cdot  x^\alpha=\pinitial{(-1,w)}{s'},
         \text{ where } \beta_0=\min\{\beta\in\NN\mid c_\beta\neq 0\}.
       \end{displaymath}
       And because the valued weighted degree in $\mathcal O_K [x]$ and the weighted degree in $\Rtx$ coincide,
       \begin{displaymath}
         \deg_w(x^\alpha)-\val(\textstyle\sum_\beta c_\beta p^\beta) = \deg_{(-1,w)}(\textstyle\sum_\beta c_\beta \cdot t^\beta x^\alpha),
       \end{displaymath}
       this implies any $f\in \mathcal O_K[ x]$ has a preimage $f'\in \Rtx$ under $\pi$ such that
       \begin{displaymath}
         \initial_{\nu,w}(f)=\pinitial{(-1,w)}{f'},
       \end{displaymath}
       simply by applying the above argument to each of its terms.
     \item[$\subseteq$]
       Once again consider a term $s=\sum_\beta c_\beta p^\beta \cdot  x^\alpha\in \mathcal O_K[x]$ with $p\nmid c_\beta$ for all $\beta\in\NN$.
       Then any preimage of it under $\pi$ is of the form $s'=\sum_\beta c_\beta t^\beta x^\alpha + r$
       for some $r\in \langle t-p\rangle$.

       If $\deg_{(-1,w)}(r) > \deg_{(-1,w)}(\sum_\beta c_\beta t^\beta x^\alpha)$,
       we would have
       \begin{displaymath}
         \pinitial{(-1,w)}{s'}=\pinitial{(-1,w)}{r}=0,
       \end{displaymath}
       since $\initial_{(-1,w)}(r)\in\initial_{(-1,w)}\langle p-t\rangle = \langle p\rangle$.

       And if $\deg_{(-1,w)}(r) < \deg_{(-1,w)}(\sum_\beta c_\beta t^\beta x^\alpha)$,
       we would have
       \begin{align*}
         \pinitial{(-1,w)}{s'}&=\ov{\initial_{(-1,w)}(\textstyle\sum_\beta c_\beta t^\beta x^\alpha)}|_{t=1}
         =\ov c_{\beta_0} \cdot  x^\alpha \\
         &=\initial_{\nu,w}(\textstyle\sum_\beta c_\beta p^\beta\cdot  x^\alpha)= \initial_{\nu,w}(s),
       \end{align*}
       where $\beta_0:=\min\{\beta\in\NN\mid c_\beta\neq 0\}$.

       Now suppose $\deg_{(-1,w)}(r) = \deg_{(-1,w)}(\sum_\beta c_\beta t^\beta x^\alpha)$.
       First observe that because $t$ is weighted negatively, there can
       be no cancellation amongst the highest weighted terms of $r$ and
       the terms of $\sum_\beta c_\beta t^\beta x^\alpha$,
       as the terms of $\sum_\beta c_\beta t^\beta x^\alpha$ are not divisible by $p$, unlike
       the terms of the highest weighted terms of $r$.
       Therefore, we have
       \begin{align*}
         \pinitial{(-1,w)}{s'}&=\underbrace{\ov{\initial_{(-1,w)}(\textstyle\sum_\beta c_\beta t^\beta x^\alpha)}|_{t=1}}_{=\initial_{\nu,w}(\textstyle\sum_\beta c_\beta p^\beta\cdot  x^\alpha)}
         + \underbrace{\pinitial{(-1,w)}{r}}_{=\ov{0}} =\initial_{\nu,w}(s).
       \end{align*}

       Either way, we always have $\pinitial{(-1,w)}{s'} \in \langle\initial_{\nu,w}(s)\rangle$ for any
       arbitrary preimage $s'\in\pi^{-1}(s)$, and, as before, the same hence holds true for any arbitrary element $f\in \mathcal O_K[ x]$. \qedhere
     \end{description}
   \end{proof}

   \begin{corollary}\label{cor:dodgingValuations}
     For any ideal $I\unlhd K[x]$ and any weight vector $w\in\RR^n$, we have:
     \begin{displaymath}
       \pinitial{(-1,w)}{\pi^{-1} I}=\initial_{\nu,w}(I).
     \end{displaymath}
   \end{corollary}
   \begin{proof}
     Follows from $\initial_{\nu,w}(I)=\initial_{\nu,w}(I\cap \mathcal O_K[ x])$.
   \end{proof}

   With our previous considerations, we can define tropical varieties of ideals in $\Rtx$ and show how some of them relate to tropical varieties of ideals in $K[x]$.

   \begin{definition}[tropical variety]\label{def:tropicalVarietiesOverRings}
     For an ideal $I\unlhd\Rtx$ we define its \emph{tropical variety} to be
     \begin{displaymath}
       \Trop(I)=\overline{\{w\in\RR_{<0}\times\RR^n\mid \initial_w(I) \text{ monomial free}\}}\subseteq\RR_{\leq 0}\times\RR^n,
     \end{displaymath}
     where $\overline{(\cdot)}$ denotes the closure in the euclidean topology.
   \end{definition}

   \begin{example}\label{ex:monFreeTermFree}
     Unlike over coefficient fields, initial ideals over coefficient rings may be devoid of monomials $t^\beta x^\alpha$, $\beta\in \NN$ and $\alpha\in\NN^n$ while containing terms $c\cdot t^\beta x^\alpha$, $c\notin R^\ast$. Consequently, tropical varieties over rings need not be pure.

     Consider the principal ideal generated by $g:=x+y+2z\in\ZZ\llbracket t \rrbracket[x,y,z]$. Figure~\ref{fig:sectionTropZ} shows the intersection of its tropical variety with an affine subspace of codimension $2$. Because $g$ is homogeneous in $x,y,z$, its tropical variety is invariant under translation by $(0,1,1,1)$, and since $t$ does not occur in $g$, it is also closed under translation by $(-1,0,0,0)$. Hence, the remaining points are then uniquely determined up to symmetry.

     \begin{figure}[h]
       \centering
       \begin{tikzpicture}
         \fill[color=blue!20]
         (-2,-2) rectangle (0,0);
         \fill[color=red!20]
         (-2,0) rectangle (2,2)
         (0,-2) rectangle (2,2);
         \draw
         (0,0) -- (2,2)
         (0,0) -- (-2,0)
         (0,0) -- (0,-2);
         \fill (0,0) circle (2pt);
         \node[anchor=north east, font=\scriptsize] at (0,0) {$(-1,0,0,0)$};
         \fill (0.5,0.5) circle (2pt);
         \node[anchor=north west, font=\scriptsize] at (0.5,0.5) {$(-1,1,1,0)$};
         \node[anchor=west] (q1) at (3,0) {$\initial_w(I) = \langle x\rangle$};
         \node[anchor=north, yshift=-5] at (q1) {contains monomial};
         \node[anchor=east] (q2) at (-3,1.25) {$\initial_w(I) = \langle y\rangle$};
         \node[anchor=north, yshift=-5] at (q2) {contains monomial};
         \node[anchor=east] (q3) at (-3,-0.75) {$\initial_w(I)=\langle 2z\rangle$};
         \node[anchor=north, yshift=-5] at (q3) {monomial free};
       \end{tikzpicture}\vspace{-0.5cm}%
       \caption{$\Trop(\langle x+y+z2\rangle) \cap \{ w_t = -1, w_z=0 \}$}
       \label{fig:sectionTropZ}
     \end{figure}

     Since $\initial_{(-1,-1,-1,0)}(g)=2z$ is no monomial, the entire
     lower left quadrant is included in our tropical variety, while the
     two other maximal cones are not. However, because
     $\initial_{(-1,1,1,0)}(g)=x+y$ is no monomial either, the edge
     containing it is also part of our tropical variety. Therefore, the
     tropical variety cannot be the support of a pure polyhedral
     complex. Note, however, that $I$ is not the type of ideal we are
     interested in, i.e.~the type of ideal occurring in the following theorem.
   \end{example}

   \begin{theorem}\label{thm:bijection}
     Let $I\unlhd K[ x]$ be an ideal. The projection $\RR_{\leq 0}\times\RR^n\rightarrow\RR^n$ induces a natural bijection
     \begin{align*}
       \Trop(\pi^{-1} I) \cap (\{-1\}\times \RR^n) &\overset{\sim}{\longrightarrow} \Trop_{\hspace{-0.1cm}\nu}(I)\\
       (-1,w_1,\ldots,w_n)&\longmapsto (w_1,\ldots,w_n).
     \end{align*}
   \end{theorem}
   \begin{proof}
     For the bijection, we show that
     \begin{displaymath}
       \initial_{\nu,w}(I) \text{ monomial free } \quad \Longleftrightarrow \quad
       \initial_{(-1,w)}(\pi^{-1} I) \text{ monomial free}.
     \end{displaymath}
     \begin{description}[leftmargin=0.5em,font=\normalfont]
     \item[$\Rightarrow$]
       Assume that $\initial_{(-1,w)}(\pi^{-1} I)\unlhd\Rtx$ contains a monomial
       $t^\beta x^\alpha$.
       Then, by Corollary~\ref{cor:dodgingValuations}, we have
       $\initial_{\nu,w}(I) = \pinitial{(-1,w)}{\pi^{-1} I}$, which means
       $\initial_{\nu,w}(I)$ must contain the monomial $x^\alpha \in \mathfrak{K}[ x]$.
     \item[$\Leftarrow$]
       Assume that $\initial_{\nu,w}(I)\unlhd \mathfrak{K}[ x]$ contains a monomial $x^\alpha$.
       Then, by Corollary~\ref{cor:dodgingValuations},
       $\initial_{(-1,w)}(\pi^{-1} I)$ must contain an element of
       of the form $ x^\alpha + (t-1)\cdot r + p\cdot s$, for some $r,s\in R[t, x]$.
       Recall that $p$ lies in $\initial_{(-1,w)}(\pi^{-1} I)$,
       therefore so does $p\cdot s$, and hence we have
       $ x^\alpha + (t-1)\cdot r\in\initial_{(-1,w)}(\pi^{-1} I)$. \\
       Let $r=r_l+\ldots+r_1$ be a decomposition of $r$ into its $(-1,w)$-homogeneous layers
       with $\deg_{(-1,w)}(r_1)<\ldots<\deg_{(-1,w)}(r_l)$.
       For sake of simplicity, we now distinguish between three cases:
       \begin{description}[leftmargin=0.1cm]
       \item[\rm 1. $\deg_{(-1,w)}( x^\alpha)\geq \deg_{(-1,w)}(r_l)$]
         Set $g_1:=r-r_1=r_l+\ldots+r_2$. Then
         \begin{displaymath}
           x^\alpha + (t-1)\cdot r =  x^\alpha + (t-1)\cdot (g_1+r_1)
           = \underbrace{ x^\alpha + (t-1)\cdot g_1-r_1}_{\text{higher weighted degree}}
           +\;t \cdot r_1.
         \end{displaymath}
         Hence $x^\alpha + (t-1)\cdot g_1-r_1,t\cdot r_1\in \initial_{(-1,w)}(\pi^{-1} I)$ and,
         more importantly,
         \begin{displaymath}
           t\cdot ( x^\alpha + (t-1)\cdot g_1 - r_1)+t\cdot r_1=t x^\alpha+(t-1) t\cdot g_1
           \in \initial_{(-1,w)}(\pi^{-1} I),
         \end{displaymath}
         effectively shaving off the $r_1$ layer.
         We can continue this process by setting $g_2:=g_1-r_2 = r_l+\ldots+r_3$.
         Then
         \begin{align*}
           t  x^\alpha + (t-1) t\cdot g_1 &= t  x^\alpha + (t-1) t\cdot (g_2+r_2) \\
           &= \underbrace{t x^\alpha + (t-1) t\cdot g_2 - t\cdot r_2}_{\text{higher weighted degree}}
           + t^2\cdot r_2.
         \end{align*}
         Hence $t x^\alpha + (t-1) t\cdot g_2-t\cdot r_2, t^2\cdot r_2\in
         \initial_{(-1,w)}(\pi^{-1} I)$ and, as above,
         \begin{align*}
           &t\cdot (t x^\alpha + (t-1) t\cdot g_2-t\cdot r_2)+t^2\cdot r_2 \\
           &\qquad   =t^2 x^\alpha+(t-1) t^2\cdot g_2 \in \initial_{(-1,w)}(\pi^{-1} I)
         \end{align*}
         removing the $r_2$ layer. Eventually, we obtain $t^l x^\alpha \in \initial_{(-1,w)}(\pi^{-1} I)$.
       \item[\rm 2. $\deg_{(-1,w)}( x^\alpha)\leq \deg_{(-1,w)}(r_1)$]
         Set $g_1:=r-r_l=r_{l-1}+\ldots+r_1$. Then
         \begin{displaymath}
           x^\alpha + (t-1)\cdot r =  x^\alpha + (t-1)\cdot (g_1+r_l)
           = \underbrace{ x^\alpha + (t-1)\cdot g_1+t\cdot r_l}_{\text{lower weighted degree}}
           - r_l.
         \end{displaymath}
         Thus $r_l,  x^\alpha + (t-1)\cdot g_1+t\cdot r_l\in \initial_{(-1,w)}(\pi^{-1} I)$ and,
         more importantly,
         \begin{displaymath}
           x^\alpha + (t-1)\cdot r_1+t\cdot g_1-t\cdot g_1= x^\alpha+(t-1)\cdot r_1
           \in \initial_{(-1,w)}(\pi^{-1} I),
         \end{displaymath}
         shaving off the the $r_l$ layer this time. Continuing this pattern eventually yields
         $ x^\alpha\in\initial_{(-1,w)}(\pi^{-1} I)$.
       \item[\rm 3. $\deg_{(-1,w)}(r_1)<\deg_{(-1,w)}( x^\alpha)<\deg_{(-1,w)}(r_l)$]
         In this case we can use a combination of the steps in the previous cases to see
         $t^i\cdot x^\alpha \in \initial_{(-1,w)}(\pi^{-1} I)$ for the $1\leq i\leq k$ such that
         $\deg_{(-1,w)}(r_{i-1})<\deg_{(-1,w)}( x^\alpha)\leq \deg_{(-1,w)}(r_i)$.
       \end{description}
       In either case, we see that $\initial_{(-1,w)}(\pi^{-1} I)$ contains a monomial. \qedhere
     \end{description}
   \end{proof}

   \begin{corollary}\label{cor:structureTheorem}
     If $I\unlhd K[x]$ defines an irreducible subvariety of $(K^\ast)^n$ of dimension $d$,
     then $\Trop(\pi^{-1} I)$ is the support of a pure polyhedral fan of dimension $d+1$ connected in codimension one.
   \end{corollary}
   % \begin{proof}
   %   By Theorem \ref{thm:structureTheorem}, there exists a pure polyhedral complex $\Delta$ in $\RR^n$ of dimension $d$ such that
   %   \begin{displaymath}
   %     \Trop_\nu(I)=\textstyle\bigcup_{\sigma\in\Delta} \sigma.
   %   \end{displaymath}
   %   Theorem \ref{thm:bijection} now implies for the lower open halfspace,
   %   \begin{displaymath}
   %     \Trop(I)\cap (\RR_{<0}\times\RR^n) = \left(\textstyle\bigcup_{\sigma\in\Delta} \Cone(\{-1\}\times\sigma)\right)\cap (\RR_{<0}\times\RR^n),
   %   \end{displaymath}
   %   where $\Cone(\{-1\}\times\sigma)$ stands for the polyhedral cone
   %   over the origin which is spanned by all points of
   %   $\{-1\}\times\sigma\subseteq \{-1\}\times\RR^n$. These cones
   %   generate a pure polyhedral fan of dimension $d+1$ in $\RR_{\leq
   %     0}\times\RR^n$, and taking the closure on both sides now yields
   %   \begin{displaymath}
   %     \Trop(I) = \textstyle\bigcup_{\sigma\in\Delta} \Cone(\{-1\}\times\sigma). \qedhere
   %   \end{displaymath}
   % \end{proof}

   \begin{example}\label{ex:puiseuxSeries}
     Let $K:=\QQ(\!(u)\!)$ be the field of Laurent series, equipped with is natural valuation $\nu_u$, and
     let $I\unlhd K [x,y]$ be the principal ideal generated by
     $(x+y+1)\cdot(u^2x+y+u)$. Then $\Trop_{\hspace{-0.1cm}\nu_u}(I)$ is the union of two tropical
     lines, one with vertex at $(0,0)$ and one with vertex at $(1,-1)$.
     Setting $R:=\QQ[t]\subseteq\QQ\llbracket t\rrbracket=\mathcal O_K$, Proposition~\ref{prop:initial} implies that for any
     weight vector $w=(w_t,w_x,w_y)\in\RR_{<0}\times\RR^2$ in the lower
     open halfspace we have
     \begin{displaymath}
       w\in\Trop(\pi^{-1}I) \quad \Longleftrightarrow \quad \left(\frac{w_x}{|w_t|},\frac{w_y}{|w_t|}\right) \in \Trop_{\hspace{-0.1cm}\nu_u}(I).
     \end{displaymath}
     Hence $\Trop(\pi^{-1}I)$ is as shown in Figure
     \ref{fig:TropPuiseux}, the cone over $\Trop_{\hspace{-0.1cm}\nu_u}(I)$. The polyhedral complex consists of $6$ rays
     and $8$ two-dimensional cones in a way that the intersection with
     the affine hyperplane yields a highlighted polyhedral complex,
     $\Trop_{\hspace{-0.1cm}\nu_u}(I)$.
     \begin{figure}[h]
       \centering
       \begin{tikzpicture}
         \fill[blue!20] (-2.5,-1,5) -- (3,-1,5) -- (3,-1,-2) -- (-2.5,-1,-2) -- cycle;
         \node[anchor=base west, font=\scriptsize, xshift=0.25cm] at (3,-1,5) {$\{-1\}\times\RR^2$};
         \draw[red] (-2,-1,0) -- (0,-1,0)
         (0,-1,0) -- (0,-1,1)
         (0,-1,0) -- (1.5,-1,-1.5);
         \draw[->,shorten >=-3em] (0,0,0) -- (-1,0,0);
         \draw[->,shorten >=-3em] (0,0,0) -- (1,0,-1);
         \draw[->,shorten >=-3em] (0,0,0) -- (0,0,1);
         \draw[->,shorten >=0.1cm] (0,0,0) -- (1,-1,1);
         \draw[->,shorten >=0.1cm] (0,0,0) -- (0,-1,0);
         \draw[->,shorten >=0.1cm] (0,0,0) -- (0,-1,1);
         \draw[shorten <=0.1cm, dashed] (1,-1,1) -- (1.5,-1.5,1.5);
         \draw[shorten <=0.1cm, dashed] (0,-1,0) -- (0,-1.75,0);
         \draw[shorten <=0.1cm, dashed] (0,-1,1) -- (0,-1.6,1.6);

         \fill (0,0,0) circle (2pt);
         \node[anchor=south,font=\scriptsize] at (0,0,0) {$(0,0,0)$};
         \draw[red] (-1,-1,1) -- (1,-1,1)
         (1,-1,1) -- (1,-1,4)
         (1,-1,1) -- (2.5,-1,-0.5);
         \draw[red] (0,-1,1) -- (0,-1,3);
         \fill[red] (1,-1,1) circle (2pt)
         (0,-1,0) circle (2pt)
         (0,-1,1) circle (2pt);
       \end{tikzpicture}\vspace{-0.5cm}%
       \caption{$\Trop(\pi^{-1}I)$ as cone over $\Trop_{\hspace{-0.1cm}\nu_u}(I)$}
       \label{fig:TropPuiseux}
     \end{figure}
   \end{example}

   \begin{example}\label{ex:chanTwoAdic}
     Consider
     $I=\langle x_1-2x_2+3x_3,3x_2-4x_3+5x_4\rangle\unlhd\QQ_2[x_1,\ldots,x_4]$,
     whose preimage is given by
     \begin{displaymath}
       \pi^{-1}I=\langle x_1-2x_2+3x_3,3x_2-4x_3+5x_4, 2-t \rangle\unlhd\ZZ\llbracket t\rrbracket[x_1,\ldots,x_4].
     \end{displaymath}
     The tropical variety of the preimage is combinatorially of the form shown in
     Figure~\ref{fig:exBijectionChan0} and is invariant under the one-dimensional subspace
     generated by $(0,1,1,1,1)$. Hence each of the six vertices represents a
     two-dimensional cone and each of the five edges represents a
     three-dimensional cone.
     \begin{figure}[h]
       \centering
       \begin{tikzpicture}
         \draw (0,0) -- (2,0);
         \fill (0,0) circle (2pt)
         (2,0) circle (2pt)
         (4,1) circle (2pt)
         (4,-1) circle (2pt)
         (-2,1) circle (2pt)
         (-2,-1) circle (2pt);
         \draw (0,0) -- (-2,1)
         (0,0) -- (-2,-1)
         (2,0) -- (4,1)
         (2,0) -- (4,-1);
         \node[anchor=south west,font=\scriptsize,xshift=-5,yshift=1em] (v1) at (0,0) {$(-1,1,-1,1,-1)$};
         \node[anchor=north east,font=\scriptsize,xshift=5,yshift=-1em] (v2) at (2,0) {$(-2,-1,1,-1,1)$};
         \node[anchor=east,font=\scriptsize] at (-2,1) {$(0,-3,1,1,1)$};
         \node[anchor=east,font=\scriptsize] at (-2,-1) {$(0,1,1,-3,1)$};
         \node[anchor=west,font=\scriptsize] at (4,1) {$(0,1,-3,1,1)$};
         \node[anchor=west,font=\scriptsize] at (4,-1) {$(0,1,1,1,-3)$};
         \draw[->,shorten >= 0.5em,thin] (v1) -- (0,0);
         \draw[->,shorten >= 0.5em,thin] (v2) -- (2,0);
       \end{tikzpicture}\vspace{-0.5cm}
       \caption{$\Trop(\langle x_1-2x_2+3x_3,3x_2-4x_3+5x_4,2-t\rangle)$}
       \label{fig:exBijectionChan0}
     \end{figure}

     Intersected with the affine hyperplane $\{-1\}\times\RR^4$, we
     obtain a polyhedral complex as shown in Figure
     \ref{fig:exBijectionChan1}, any vertex of Figure~\ref{fig:exBijectionChan0} in $\{0\}\times\RR^4$ becoming a point at infinity.

     \begin{figure}[h]
       \centering
       \begin{tikzpicture}
         \draw (0,0) -- (2,0);
         \fill (0,0) circle (2pt);
         \fill (2,0) circle (2pt);
         \draw (0,0) -- node[sloped] (a) {$<$} (-2,1)
         (0,0) -- node[sloped] (b) {$<$} (-2,-1)
         (2,0) -- node[sloped] (c) {$>$} (4,1)
         (2,0) -- node[sloped] (d) {$>$} (4,-1);
         \node[anchor=south west,font=\scriptsize,xshift=-5,yshift=1em] (v1) at (0,0) {$(1,-1,1,-1)$};
         \node[anchor=north east,font=\scriptsize,xshift=5,yshift=-1em] (v2) at (2,0) {$\frac{1}{2}(-1,1,-1,1)$};
         \node[anchor=east,font=\scriptsize, xshift=-0.2cm] at (a) {$(-3,1,1,1)$};
         \node[anchor=east,font=\scriptsize, xshift=-0.2cm] at (b) {$(1,1,-3,1)$};
         \node[anchor=west,font=\scriptsize, xshift=0.2cm] at (c) {$(1,-3,1,1)$};
         \node[anchor=west,font=\scriptsize, xshift=0.2cm] at (d) {$(1,1,1,-3)$};
         \draw[->,shorten >= 0.5em,thin] (v1) -- (0,0);
         \draw[->,shorten >= 0.5em,thin] (v2) -- (2,0);
       \end{tikzpicture}\vspace{-0.5cm}
       \caption{$\Trop_{\hspace{-0.1cm}\nu_2}(\langle x_1-2x_2+3x_3,3x_2-4x_3+5x_4\rangle)$}
       \label{fig:exBijectionChan1}
     \end{figure}

   \end{example}

   %%%%%%%%%%%%%%%%%%%%%%%%%%%%%%%%%%%%%%%%%%%%%%%%%%%%%%%%%%%%%%%%%%%%%%%%%%%%%%%%%%

   \section{Tracing Gr\"obner complexes to a trivial valuation}\label{sec:groebnerComplex}

   In this section, we show how the Gr\"obner complexes of ideals in $K[x]$ can be traced back to the Gr\"obner fans of ideals in $\Rtx$. We will show how the Gr\"obner fan induces a refinement of the Gr\"obner complex and how to determine whether two integral Gr\"obner cones map to the same valued Gr\"obner polytope. For the latter, we will need to delve into some basics in Gr\"obner bases. We close this section with a remark on $p$-adic Gr\"obner bases as introduced by Chan and Maclagan \cite{CM13}.

   \begin{definition}[Gr\"obner polyhedra, Gr\"obner complexes over valued fields]
     % Let $I\unlhd K[x]$ be a homogeneous ideal and fix a weight vector $w\in\RR^n$. We call a finite subset $G\subseteq I$ a \emph{Gr\"obner basis} with respect to $w$, if $\initial_{\nu,w}(I) = \langle \initial_{\nu,w}(g)\mid g\in G\rangle$.
     For a homogeneous ideal $I\unlhd K[x]$ and a weight vector $w\in\RR^n$ we define its \emph{Gr\"obner polytope} to be
     \begin{displaymath}
       C_{\nu,w}(I):=\overline{\{v\in\RR^n\mid \initial_{\nu,v}(I) = \initial_{\nu,w}(I) \}}\subseteq \RR^n,
     \end{displaymath}
     where $\ov{(\cdot)}$ denotes the closure in the euclidean topology.
     We will refer to the collection $\Sigma_\nu(I):=\{C_{\nu,w}(I)\mid w\in\RR^n\}$ as the \emph{Gr\"obner complex} of $I$.
   \end{definition}

   \begin{theorem}[Gr\"obner complex, {\cite[Theorem 2.5.3]{MS15}}]
     Let $I\unlhd K[x]$ be a homogeneous ideal. Then all $C_{\nu,w}(I)$ are convex polytopes and $\Sigma_\nu(I)$ is a finite polyhedral complex.
   \end{theorem}

   \begin{definition}
     For an $x$-homogenous ideal $I\unlhd\Rtx$, i.e.~an ideal
     generated by elements which are homogeneous if considered as
     polynomials in $x$ with coefficients in $\Rt$, and a weight vector $w\in\RR_{<0}\times\RR^n$ we define its Gr\"obner cone to be
     \[ C_w(I):=\overline{\{v\in\RR_{<0}\times\RR^n\mid \initial_v(I)=\initial_w(I)\}}, \]
     where $\ov{(\cdot)}$ denotes the closure in the euclidean topology.
     We will refer to the collection $\Sigma(I):=\{C_w(I), C_w(I)\cap \{0\}\times\RR^n \mid w\in\RR_{<0}\times \RR^n\}$ as the \emph{Gr\"obner fan} of $I$.
   \end{definition}

   \begin{proposition}[\cite{MarkwigRenGroebnerFan}, Theorem 3.19]
     Let $I\unlhd\Rtx$ be an $x$-homogeneous ideal. Then all $C_w(I)$ are polyhedral cones and $\Sigma(I)$ is a finite polyehdral fan.
   \end{proposition}

   \begin{corollary}
     The map $\{-1\}\times \RR^n\stackrel\sim\longrightarrow\RR^n, (-1,w)\longmapsto w$ is compatible with the Gr\"obner fan $\Sigma(\pi^{-1}I)$ and the Gr\"obner complex $\Sigma_\nu(I)$, i.e. it maps the restriction of a Gr\"obner cone $C_{(-1,w)}(\pi^{-1}I)\cap \big(\{-1\}\times\RR^n\big)$ into a Gr\"obner polytope $C_{\nu,w}(I)$.
   \end{corollary}
   \begin{proof}
     Follows directly from Proposition \ref{prop:initial}.
   \end{proof}

   Note that it may very well happen that several cones are mapped
   into the same Gr\"obner polytope, i.e.~that the image of the
   restricted Gr\"obner fan is a refinement of the Gr\"obner complex
   (see Example~\ref{ex:groebnerpolytope}).

   We will now recall the notion of initially reduced standard bases of ideals in $\Rtx$ from \cite{MarkwigRenStandardBases} and how they determine the inequalities and equations of Gr\"obner cones as shown in \cite{MarkwigRenGroebnerFan}. We will then use them to decide whether two Gr\"obner cones are mapped to the same Gr\"obner polytope and, by doing so, show that no separate standard basis computation is required for it.

   \begin{definition}[initially reduced standard bases]\label{def:initialReduction}
     Fix the $t$-local lexicographical ordering $>$ such that $x_1>\ldots>x_n>1>t$.

     Given a weight vector $w\in\RR_{<0}\times\RR^n$ we define the \emph{weighted ordering} $>_w$ to be
     \begin{align*}
       t^\beta x^\alpha >_w t^\delta x^\gamma \quad :\Longleftrightarrow \quad &w\cdot (\beta,\alpha) > w\cdot (\delta,\gamma) \text{ or}\\
       & w\cdot (\beta,\alpha)=w\cdot (\delta,\gamma) \text{ and } t^\beta x^\alpha > t^\delta x^\gamma.
     \end{align*}

     For $g\in\Rtx$, the \emph{leading term} $\lt_{>_w}(g)$ is the unique term of $g$ with maximal monomial under $>_w$ and for $I\unlhd\Rtx$, the \emph{leading ideal} $\lt_{>_w}(I)$ is the ideal generated by the leading terms of all its elements. A finite subset $G\subseteq I$ is called a \emph{standard basis} of $I$ with respect to $>_w$, if the leading terms of its elements generate $\lt_{>_w}(I)$.

     Suppose $G=\{g_1,\ldots,g_k\}$ with $g_i=\sum_{\alpha\in\NN^n} g_{i,\alpha}\cdot x^\alpha$, $g_{i,\alpha} \in \Rt$. We call $G$ \emph{initially reduced}, if the set
     \begin{displaymath}
       G':=\Big\{\sum_{\alpha\in\NN} \lt_>(g_{i,\alpha})\cdot x^\alpha\;\bigmid\; i=1,\ldots,k\Big\}\subseteq R[t,x],
     \end{displaymath}
     is reduced in the classical sense.
   \end{definition}

   \begin{proposition}[{\cite[Algorithm 4.6]{MarkwigRenGroebnerFan}}]\label{prop:StandardBases}
     Let $I\unlhd\Rtx$ be an $x$-homogeneous ideal and $w\in\RR_{<0}\times\RR^n$ a weight vector.
     Then an initially reduced standard basis $G$ of $I$ with respect to $>_w$ exists.

     Moreover, if $I$ can be generated by elements in $R[t,x]$, then $G$ can be computed in finite time.
   \end{proposition}

   \begin{proposition}[{\cite[Proposition 3.8, 3.11]{MarkwigRenGroebnerFan}}]\label{prop:GroebnerFan}
     Let $I\unlhd\Rtx$ be an $x$-homogeneous ideal, let $w\in\RR_{<0}\times\RR^n$ be a weight vector and
     let $G$ an initially reduced standard basis of $I$ with respect
     to $>_w$. Then the set of its initial forms $\{\initial_w(g)\mid
     g\in G\}$ is an initially reduced standard basis of
     $\initial_w(I)$ with respect to $>_w$, and the Gr\"obner cone of
     $I$ around $w$ is given by
     \[ C_w(I)=\overline{\{v\in\RR_{<0}\times\RR^n\mid \initial_v(g)=\initial_w(g) \text{ for all } g\in G\}}. \]
   \end{proposition}

   We now show that our standard bases of $\pi^{-1}I\unlhd \Rtx$ yield Gr\"obner bases of initial ideals of $I\unlhd K[x]$, allowing us to immediately decide whether two Gr\"obner cones of the former are mapped to the same Gr\"obner polytope of the latter.

   \begin{corollary}\label{cor:StandardBases}
     Let $I\unlhd K[x]$ be a homogeneous ideal, let $w\in \RR^n$ be a weight vector and let $G$ be an initially reduced standard basis of $\pi^{-1}I$ with respect to the weighted ordering $>_{(-1,w)}$. Then
     \[ \left\{\overline{\initial_{(-1,w)}(g)}|_{t=1}\suchthat g\in G\right\}\]
     is a standard basis of $\initial_{\nu,w}(I)$ with respect to the fixed lexicographical ordering $>$ restricted to monomials in $x$.
   \end{corollary}
   \begin{proof}
     By Proposition~\ref{prop:GroebnerFan}, the set $\initial_{(-1,w)}(G):=\{\initial_{(-1,w)}(g)\mid g\in G\}$ is an initially reduced standard basis of $\initial_{(-1,w)}(\pi^{-1}I)$ with respect to $>_{(-1,w)}$. And because it is homogeneous with respect to weight vector $(-1,w)$, it is also an initially reduced standard basis with respect to $>$. By choice of $>$, the set $\initial_{(-1,w)}(G)|_{t=1}$ remains a standard basis of $\initial_{(-1,w)}(\pi^{-1}I)|_{t=1}$ with respect to the restriction of $>$ to monomials in $x$. And since $p\in \initial_{(-1,w)}(G)|_{t=1}$, $\pinitial{(-1,w)}{G}$ is a standard basis of $\pinitial{(-1,w)}{\pi^{-1}I}$ with respect to the restriction of $>$.
   \end{proof}

   \begin{example}\label{ex:groebnerpolytope}
     Consider the preimage $\pi^{-1}I\unlhd\ZZ\llbracket t\rrbracket[x,y,z]$ of the ideal $I=\langle 2y+x,z^2+y^2\rangle\unlhd\QQ_2[x,y,z]$ and the two weight vectors $w=(1,3,7), v=(1,10,5)\in\RR^3$. Fix a lexicographical tiebreaker $>$ with $x>y>z>1>t$.

     The initially reduced standard basis of $\pi^{-1}I$ under $>_{(-1,w)}$ and $>_{(-1,v)}$ are the following two sets respectively (initial forms underlined):
     \begin{align*}
       G_{(-1,w)} =\{\underline 2-t,\underline{ty}+x,\underline{z^2}+y^2\}, \;
       G_{(-1,v)} =\{\underline 2-t,\underline{ty}+x, \underline{xy}-tz^2, \underline{t^2z^2}+x^2, \underline{y^2}+z^2\},
     \end{align*}
     yielding the following Gr\"obner basis of $\initial_{\nu,w}(I)$ and $\initial_{\nu,v}(I)$ under $>$:
     \begin{align*}
       \mathcal G_w=\{y, z^2\}, \quad \mathcal G_v=\{y,xy,z^2,y^2\}.
     \end{align*}
     One immediately sees that both initial ideals coincide, meaning that the two Gr\"obner cones $C_{(-1,w)}(\pi^{-1}I)$ and $C_{(-1,v)}(\pi^{-1}I)$ are mapped to the same Gr\"obner polytope $C_{\nu_2,w}(I)=C_{\nu_2,v}(I)$.
   \end{example}

   \begin{remark}[homogenization and dehomogenization]\label{rem:homogenization}
     A lot of effort has been put into developing algorithms for
     computing Gr\"obner cones $C_w(I)$ for $x$-homogeneous ideals
     $I\unlhd\Rtx$ and weight vectors $w\in\RR_{<0}\times\RR^n$ in
     \cite{MarkwigRenGroebnerFan} which terminate in finite time in
     case $I$ can be generated by elements in $R[t,x]$,  avoiding the
     necessity to use homogenization and dehomogenization techniques
     as described in \cite[Lemma 1.1]{BJSST07}, which are known to
     refine the Gr\"obner fan structure in general.

     The most prominent phenomenon showing the refinement is the
     non-regular Gr\"obner fan in \cite[Theorem 1]{Jensen07b}. Note
     that Gr\"obner fans of homogeneous ideals are known to be
     regular, as they are the normal fans of the state polytopes
     \cite[Theorem 2.5]{Sturmfels96}. The non-regular Gr\"obner fan
     $\Sigma(I)$ arises from the inhomogeneous ideal
     \[ I:=\langle x_1x_3x_4+x_1^2x_3-x_1x_2,x_1x_4^2-x_3,x_1x_4^4+x_1x_3 \rangle \unlhd\QQ[x_1,\ldots,x_4] \]
     and hence is restricted to the positive orthant $\RR^4_{\geq
       0}$. However, once homogenized it yields a regular Gr\"obner
     fan $\Sigma(I^h)$ living in $\RR^5$, whose restriction to
     $\{0\}\times\RR^4_{\geq 0}$ refines $\Sigma(I)$.
   \end{remark}

   \begin{remark}[$p$-adic Gr\"obner bases]\label{rem:PadicGroebnerBases}
     A Gr\"obner basis of an ideal $I\unlhd K[x]$ over valued fields with respect to a
     weight vector $w\in\RR^n$ is by \cite[Section
     2.4]{MS15} a finite generating set whose initial forms generate
     the initial ideal $\initial_{\nu,w}(I)$.
     Observe that Corollary~\ref{cor:StandardBases} implies that such a
     Gr\"obner basis can be computed by projecting an initially reduced standard basis
     of $\pi^{-1}I\unlhd\Rtx$ under the monomial ordering $>_w$ via
     $\pi$ to $K[x]$.

     Figure~\ref{fig:StandardBasesTimings} shows timings of the \textsc{Macaulay2} Package \texttt{GroebnerValuations} from Andrew Chan \cite{M2,CM13}, a toy-implementation of a $p$-adic Matrix-F5 algorithm by Tristan Vaccon in \textsc{Sage} \cite{sagemath,Vaccon2014} and the standard basis engine of \textsc{Singular} over integers under mixed orderings \cite{DGPS16}. The examples are:
     \begin{description}[leftmargin=*]
     \item[Cyclic(n)] In $\QQ_2[x_0,\ldots,x_n]$, the cyclic ideal in the variables $x_1,\ldots,x_n$, homogenized using the variable $x_0$, and weight vector $(1,\ldots,1)$.
     \item[Katsura(n)] In $\QQ_2[x_0,\ldots,x_n]$, the Katsura ideal in the variables $x_1,\ldots,x_n$, homogenized using the variable $x_0$, and weight vector $(1,\ldots,1)$.
     \item[Chan] In $\QQ_3[x_0,\ldots,x_n]$, the ideal $\langle 2x_1^2+3x_1x_2+24x_3x_4, 8x_1^3+x_2x_3x_4+18x_3^2x_4\rangle$ and weight vector $(-1,-11,-3,-19)$ taken from \cite[Chapter 3.6]{Chan13}.
     \end{description}
     All computations were aborted after exceeding either 1 CPU day or 16 GB memory.
     Note that the computations in \textsc{Sage} were done up to a finite precision of $p^{50}$ and that the correctness of the result could only be verified for the examples for which either \textsc{Macaulay2} or \textsc{Singular} finished.
     \begin{figure}[h]
       \centering
       \begin{tabular}{l|c|c|c}%
         Examples    & \textsc{Macaulay2} & \textsc{Sage} & \textsc{Singular} \\ \hline
         Cyclic(4)   & 1                  & 10            & 1 \\
         Cyclic(5)   & -                  & -             & 1 \\
         Cyclic(6)   & -                  & -             & - \\
         Katsura(3)  & 1                  & 1             & 1 \\
         Katsura(4)  & -                  & 10            & 1 \\
         Katsura(5)  & -                  & 190           & 1 \\
         Katsura(6)  & -                  & 2900          & - \\
         Chan        & 1                  & 4             & - \\
         % Grass(2,6)  & 1  & 1 & 1 \\
         % Grass(2,7)  & 2  & 1 & 1 \\
         % Grass(2,8)  & 12 & 1 & 1 \\
         % Grass(2,9)  & 90 & 7 & 1 \\
         % Grass(2,10) & 1h & 140 & 1 \\
         % Grass(2,11) & -  & 500 & 5 \\
       \end{tabular}\vspace{-0.25cm}
       \caption{Timings in seconds unless aborted}
       \label{fig:StandardBasesTimings}
     \end{figure}
   \end{remark}

   %%%%%%%%%%%%%%%%%%%%%%%%%%%%%%%%%%%%%%%%%%%%%%%%%%%%%%%%%%%%%%%%%%%%%%%%%%%%%%%%%%

   \section{Computation of tropical varieties}\label{sec:computation}

   In this section, we present an algorithm for computing the tropical variety of an $x$-homogeneous ideal $I\unlhd \Rtx$, provided it is pure and connected in codimension one, as is the case for all preimages of ideals in $K[x]$ under $\pi$.
   All algorithms in this section are straight-forward modification of the techniques developed by Bogart, Jensen, Speyer, Sturmfels and Thomas for tropical varieties of homogeneous polynomial ideals over ground fields with trivial valuation, which is why proofs are omitted and instead references to \cite{BJSST07} are added.

   Before we begin, we quickly note that the computation of tropical hypersurfaces is simple:
   \begin{algorithm}[TropHypersurface, {\cite[Algorithm 4.3]{BJSST07}}]\label{alg:tropElement}\
     \begin{algorithmic}[1]
       \REQUIRE{$g=\sum_{\beta,\alpha} c_{\alpha,\beta} \cdot t^\beta x^\alpha$, $g\neq 0$.}
       \ENSURE{$\Delta$, collection of maximal dimensional cones in $\RR_{\leq 0}\times\RR^n$ such that
         \begin{center}
           $\Trop(g):=\Trop(\langle g\rangle)=\textstyle\bigcup_{\sigma\in\Delta} \sigma.$
         \end{center}}
       \STATE Construct the finite set of exponent vectors with minimal entry in $t$,
       \begin{center}
         $\Lambda:=\left\{(\beta,\alpha)\in\NN\times\NN^n\suchthat \begin{array}{c} \alpha \in \NN^n \text{ with } c_{\alpha,\beta'}\neq 0 \text{ for some } \beta'\in\NN \\
             \beta = \min \{\beta'\in\NN\mid c_{\alpha,\beta'}\neq 0\} \end{array} \right\}.$
       \end{center}
       \STATE Construct the normal fan of its convex hull
       \begin{center}
         $\Delta:=\text{NormalFan}(\Conv(\Lambda)).$
       \end{center}
       \RETURN{$\{\sigma\in \Delta\mid \sigma \cap \RR_{<0}\times\RR^n\neq\emptyset \text{ and }\dim(\sigma)=n\}$.}
     \end{algorithmic}
   \end{algorithm}

   The computation of general tropical varieties on the other hand
   works in three steps:
   \begin{enumerate}[leftmargin=*]
   \item Finding a first maximal Gr\"obner cone $C_w(I)\subseteq\Trop(I)$, Alg.~\ref{alg:tropStartingCone}.
   \item Given $C_u(I)\subseteq\Trop(I)$ of codimension one, describe $\Trop(I)$ around $C_u(I)$, Alg.~\ref{alg:tropCurve}.
   \item Given $C_w(I)\subseteq\Trop(I)$ maximal, compute an adjacent $C_v(I)\subseteq\Trop(I)$, Alg.~\ref{alg:flip}.
   \end{enumerate}
   of which (3) is a generalisation of the well-known \textit{flip} of Gr\"obner bases, which we will simply cite from \cite{MarkwigRenGroebnerFan} without going into any algorithmic details:

   \pagebreak
   \begin{algorithm}[Flip, {\cite[Algorithm 5.5]{MarkwigRenGroebnerFan}}]\label{alg:flip}\
     \begin{algorithmic}[1]
       \REQUIRE{$(G,H,v,>_w)$, where
         \begin{itemize}[leftmargin=*, itemsep=0pt]
         \item $>_w$ a weighted monomial ordering with weight vector $w\in\RR_{<0}\times\RR^n$,
         \item $v$ an outer normal vector of $C_w(I)$,
         \item $G=\{g_1,\ldots,g_k\}\subseteq I$ an initially reduced standard basis of an $x$-homogeneous ideal $I$ w.r.t.~$>_w$,
         \item $H=\{h_1,\ldots,h_k\}$ with $h_i=\initial_w(g_i)$.
         \end{itemize}}
       \ENSURE{$(G',>_{w'})$, where
         \begin{itemize}[leftmargin=*]
         \item $C_{w'}(I)$ adjacent to $C_w(I)$ in direction $v$, i.e. $C_{w'}(I)=C_{w+\varepsilon\cdot v}(I)$ for $\varepsilon>0$ sufficiently small,
         \item $G'\subseteq I$ an initially reduced standard basis w.r.t.~$>_{w'}$.
         \end{itemize}}
     \end{algorithmic}
   \end{algorithm}

   To show how to find a first maximal dimensional Gr\"obner cone on
   $\Trop(I)$, we need to introduce the homogeneity space, since the
   starting cone algorithm works inductively over the codimension of
   it, and we have to recall the lift of standard bases, which we will again cite
   from \cite{MarkwigRenGroebnerFan} without going into any
   algorithmic details. The latter allows us to lift a standard basis
   of an initial ideal into a standard basis of the original ideal,
   useful for avoiding unnecessary standard basis computations.

   \begin{definition}[homogeneity space]\label{def:homogeneitySpace}
     Given an $x$-homogeneous ideal $I\unlhd \Rtx$, we define the \emph{homogeneity space} of $I$ (or of $\Trop(I)$) to be the intersection of all its lower Gr\"obner cones, i.e. Gr\"obner cones of the form $C_w(I)$ for some $w\in\RR_{<0}\times\RR^n$,
     \begin{displaymath}
       C_0(I):= \bigcap_{w\in\RR_{<0}\times\RR^n} C_w(I).
     \end{displaymath}
   \end{definition}

   \begin{example}\label{ex:homogeneitySpace}
     Note that our definition of homogeneity space $C_0(I)$ differs from the natural lineality space $C_0(I)$ of tropical varieties over fields with trivial valuation. In general, our $C_0(I)$ is neither a linear subspace nor is it the set of all vectors with respect to whom the ideal is weighted homogeneous. Consider the principal ideal
     \begin{displaymath}
       I=\langle (1+t)\cdot x+(1+t)\cdot y\rangle\unlhd\ZZ\llbracket t\rrbracket[x,y],
     \end{displaymath}
     whose Gr\"obner Fan splits the weight space $\RR_{\leq 0}\times\RR^2$ into two maximal cones, see Figure \ref{fig:homogeneitySpace}, and whose homogeneity space is given by
     \begin{displaymath}
       C_0(I)=\{(w_t,w_x,w_y)\in\RR_{\leq 0}\times\RR^n\mid w_x=w_y\}.
     \end{displaymath}
     Clearly, $C_0(I)$ is no subspace and we have $(-1,0,0)\in C_0(I)$ despite the ideal not being weighted homogeneous with respect to it. This effect is caused by the terms $tx$ and $ty$ in the generator, which do not appear in any initial form and hence have no influence on $C_0(I)$, yet still exist and thus prevent $I$ from being weighted homogeneous with respect to any weight vector in the interior of $C_0(I)$.
     \begin{figure}[h]
       \centering
       \begin{tikzpicture}
         \fill[red!30] (-3,0,-3) -- (3,0,3) -- (3,-2,3) -- (-3,-2,-3) -- cycle;
         \node[anchor=west, font=\footnotesize, red] at (3,-1,3) {$C_0(I)=\{w_x=w_y\}$};
         \fill[blue!20] (-3,0,-3) -- (3,0,-3) -- (3,0,3) -- (-3,0,3) -- cycle;
         \node[anchor=south west, font=\footnotesize, xshift=0.5cm] at (-3,0,3) {$\{0\}\times\RR^2$};
         \draw[red] (-3,0,-3) -- (3,0,3);
         \draw[red!30, dashed] (-3,0,-3) -- (-3,-2,-3) -- (-2.3,-2,-2.3);
         \node[anchor=south east, font=\footnotesize, red] at (-3,0,-3) {$\RR\cdot (0,1,1)$};
       \end{tikzpicture}\vspace{-0.5cm}%
       \caption{$C_0(\langle(1+t)\cdot x+(1+t)\cdot y\rangle)$}
       \label{fig:homogeneitySpace}
     \end{figure}
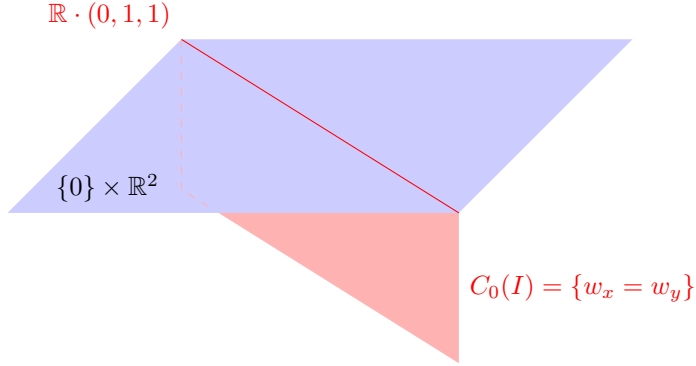
   \end{example}

   We follow up our observation in Example~\ref{ex:homogeneitySpace} with the following Lemma, which shows that the homogeneity space behaves properly in the case which is of interest to us:

   \begin{lemma}\label{lem:homogeneitySpace}
     Let $I\unlhd\Rtx$ be an $x$-homogeneous ideal and $w\in\RR_{<0}\times\RR^n$ a weight vector. Then
     \begin{displaymath}
       C_0(\initial_w(I))= \overline{\{v\in\RR_{<0}\times\RR^n \mid \initial_v\initial_w(I)=\initial_w(I)\}} = \Lin(C_w(I))\cap (\RR_{\leq0}\times\RR^n).
     \end{displaymath}
   \end{lemma}
   \begin{proof}
     The second equality follows directly from the perturbation of initial ideals, i.e. it follows from the fact that for any $v\in\RR_{<0}\times\RR^n$ we have $\initial_v\initial_w(I)=\initial_{w+\varepsilon\cdot v}(I)$ for $\varepsilon>0$ sufficiently small \cite[Proposition 5.4]{MarkwigRenGroebnerFan}. It remains to show the first equality.

     The $\supseteq$ inclusion can be shown in a similar fashion: Suppose $v\in\RR_{<0}\times\RR^n$ such that $\initial_v(\initial_w(I))=\initial_w(I)$. Then for any $u\in \RR_{<0}\times\RR^n$ we have
     \[ \initial_{v+\varepsilon \cdot u} (\initial_w(I)) = \initial_u(\initial_v(\initial_w(I)))=\initial_u(\initial_w(I)),\]
     showing that $v+\varepsilon\cdot u\in C_u(\initial_w(I))$ for any $\varepsilon>0$ sufficiently small. As $C_u(\initial_w(I))$ is closed by definition, this implies $v\in C_u(\initial_w(I))$. This shows that $v$ is contained in every lower Gr\"obner cone of $\initial_w(I)$, and hence also in their intersection $C_0(\initial_w(I))$.

     For the $\subseteq$ inclusion, consider $v\in C_0(\initial_w(I))\cap(\RR_{<0}\times\RR^n)$, so that $v\in C_u(\initial_w(I))$ for all $u\in\RR_{<0}\times\RR^n$. In particular, $v\in C_w(\initial_w(I))$ which is the middle set by definition.
   \end{proof}

   \pagebreak
   \begin{algorithm}[Lift, {\cite[Algorithm 5.2]{MarkwigRenGroebnerFan}}]\label{alg:lift}\
     \begin{algorithmic}[1]
       \REQUIRE{$(H',>',H,G,>)$, where
         \begin{itemize}[leftmargin=*, itemsep=0pt]
         \item $>$ a weighted $t$-local monomial ordering on $\Mon(t,x)$ with weight vector in $\RR_{<0}\times\RR^n$,
         \item $G=\{g_1,\ldots,g_k\}\subseteq I$ an initially reduced standard basis of an $x$-homogeneous ideal $I$ w.r.t.~$>$,
         \item $H=\{h_1,\ldots,h_k\}$ with $h_i=\initial_w(g_i)$ for
           some $w\in C_>(I)$ with $w_0<0$,
         \item $>'$ a $t$-local monomial ordering such that $w\in C_>(I)\cap C_{>'}(I)$,
         \item $H'\subseteq \initial_w(I)$ a weighted homogeneous standard basis w.r.t.~$>'$.
         \end{itemize}}
       \ENSURE{$G'\subseteq I$, an initially reduced standard basis of $I$ w.r.t.~$>'$.}
     \end{algorithmic}
   \end{algorithm}

   \begin{algorithm}[TropStartingCone, {\cite[Algorithm 4.12]{BJSST07}}]\label{alg:tropStartingCone}\
     \begin{algorithmic}[1]
       \REQUIRE{$(G,>_w)$, where $G$ is an initially reduced standard basis of an $x$-homogeneous ideal $I$ with respect a weighted ordering $>_w$, $w\in\RR_{<0}\times\RR^n$.}
       \ENSURE{$(C_{w'}(I),G',>_{w'})$, where $C_{w'}(I)\subseteq\Trop(I)$ maximal dimensional and $G'$ an initially reduced standard basis of $I$ with respect to the weighted ordering $>_{w'}$.}
       \STATE \textbf{if} $\dim(I)=\dim(C_0(I))$ \textbf{then return} $(C_0(I),G,>)$
       \STATE Find a weight vector $w\in(\Trop(I)\setminus C_0(I))\cap(\RR_{<0}\times\RR^n)$.
       \STATE Compute an initially reduced standard basis $G''$ of $I$ with respect to $>_w$.
       \STATE Set $H'':=\{\initial_w(g)\mid g\in G''\}$.
       \STATE Rerun
       \begin{center}
         $(C_{w_0'}(I),G_0',>_0')=\text{TropStartingCone}(H'',>_w).$
       \end{center}
       \STATE Let $>'$ be the weighted ordering with weight vector $w$ and tiebreaker $>_0'$.
       \STATE Lift $G_0'$ to an initially reduced standard basis $G'$ of $I$:
       \begin{center}
         $G'=\text{Lift}(G_0',>',H'',G'',>_w).$
       \end{center}
       \STATE Construct the corresponding Gr\"obner cone $C_{w'}(I) := C(H',G',>')$.
       \RETURN{$(C_{w'}(I),G',>')$}
     \end{algorithmic}
   \end{algorithm}

   \begin{example}\label{ex:ChanStartingCone}
     Let $I\unlhd \ZZ\llbracket t\rrbracket[x_1,\ldots,x_4]$ be the preimage from Example~\ref{ex:Chan},
     \[ I=\langle 3-t,2x_1^2+3x_1x_2+24x_3x_4, 8x_1^3+x_2x_3x_4+18x_3^2x_4\rangle.\]
     A short calculation reveals that $\dim(\Trop(I))=\dim(I)=3>1=\dim(C_0(I))$ with
     \begin{displaymath}
       C_0(I) = \RR\cdot (0,1,1,1,1).
     \end{displaymath}
     Picking $w:=(-2,-1,1,5,-5)\in\Trop(I)$, the initial ideal $\initial_w(I)$
     is generated by
     \begin{displaymath}
       \{3, \quad tx_3x_4-tx_1x_2+x_1^2, \quad tx_1x_2^2-x_1^2x_2-t^3x_1x_2x_3+t^2x_1^2x_3\}.
     \end{displaymath}
     Another short calculation reveals $\dim(\initial_w(I))=3>2=\dim(C_0(\initial_w(I)))$ with
     \begin{displaymath}
       C_0(\initial_w(I)) = \RR_{\geq 0}\cdot w + \RR\cdot (0,1,1,1,1).
     \end{displaymath}
     Picking $v:=(6,-11,11,-1,1)\in\Trop(I)$, the initial ideal $\initial_v\initial_w(I)$ is generated by
     \begin{displaymath}
       \{3, \quad tx_3x_4-tx_1x_2, \quad  x_2x_3x_4-t^2x_1x_2x_3 \}.
     \end{displaymath}
     And since $\dim(\initial_v \initial_w(I))=3=\dim(C_0(\initial_v \initial_w(I)))$ with
     \begin{displaymath}
       C_0(\initial_v \initial_w(I)) = \RR_{\geq 0}\cdot (-1,1,-1,1,-1) + \RR \cdot (0,1,0,0,1) + \RR\cdot (0,0,1,1,0),
     \end{displaymath}
     the recursions end.

     \begin{figure}[h]
       \centering
       \begin{tikzpicture}
         \fill[pattern=senwStripes,pattern color=blue!20] (0,0) -- (0,2) -- (4,1.33) -- cycle;
         \fill[blue!20] (0,0) -- (0,2) -- (3,1.5) -- cycle;
         \fill (0,0) circle (2pt);
         \draw[->] (0,0) -- (0,2);

         \node[anchor=north west, font=\small, xshift=-0.65cm] at (0,0) {$C_0(I)$};
         \draw[red,very thick,->] (0,0.35) -- node [red,right] {$w$} (0,1.1);
         \draw[red,very thick,->] (0,1.45) -- node [red,above] {$v$} (0.75,1.325);

         \node[anchor=east, font=\small] at (0,1) {$C_0(\initial_w(I))$};

         \node[anchor=south west, font=\small] at (1.5,1.75) {$C_0(\initial_v \initial_w(I))$};

         \node[anchor=north, font=\small] at (1.5,-0.65) {$1=\dim(C_0(I))<\dim(C_0(\initial_w(I)))<\dim(C_0(\initial_v\initial_w(I)))=\dim\Trop(I)=3$};
       \end{tikzpicture}\vspace{-0.5cm}%
       \caption{computing a tropical starting cone recursively}
       \label{fig:tropStartingCone}
     \end{figure}

     This shows that $w+\varepsilon\cdot v\in \Trop(I)$ for
     $\varepsilon > 0$ sufficiently small. Together with the one-dimensional $C_0(I)$, this determines a maximal, three-dimensional
     Gr\"obner cone in our tropical variety, see Figure~\ref{fig:tropStartingCone}.
     % What remains is to compute an initially reduced
     % standard basis of $\initial_{w+\varepsilon\cdot v}(I)$ and $I$ with
     % respect to $>$, which we will return. Using that data,
     % Algorithm~\ref{alg:groebnerConeNoWeight} constructs a maximal
     % Gr\"obner cone in our tropical variety.
   \end{example}

   Two centrals tools necessary to describe the tropical variety around one of its codimension one cells are generic weight vectors and tropical witnesses.

   \begin{definition}[multiweights and generic weights]\label{def:multiweights}
     Given weight vectors $w\in\RR_{<0}\times\RR^n$ and
     $v_1,\ldots,v_d\in \RR \times\RR^n$, we define the \emph{initial
       form} of an element $g\in \Rtx$ with respect to the multidegree
     $(w,v_1,\ldots,v_d)$ to be
     \begin{displaymath}
       \initial_{(w,v_1,\ldots,v_d)} (g) = \initial_{v_d}\ldots \initial_{v_1}\initial_{w}(g),
     \end{displaymath}
     and we define the \emph{initial ideal} of $I\unlhd \Rtx$ with respect to $(w,v_1,\ldots,v_d)$ to be
     \begin{displaymath}
       \initial_{(w,v_1,\ldots,v_d)} (I) = \initial_{v_d}\cdots \initial_{v_1}\initial_{w}(I) =\langle \initial_{(w,v_1,\ldots,v_d)} (g) \mid g\in I \rangle.
     \end{displaymath}
     Also, still fixing the lexicographical ordering $>$ with $x_1>\ldots>x_n>1>t$ from Definition~\ref{def:initialReduction}, we define the multiweighted ordering $\textcolor{red}{>_{(w,v_1,\ldots,v_k)}}$ to be
     \begin{align*}
       & t^\beta\cdot x^\alpha >_{(w,v_1,\ldots,v_k)} t^\delta\cdot x^\gamma \quad \Longleftrightarrow \quad \text{either}: \\
       & \quad \bullet\; w\cdot (\beta,\alpha) > w\cdot (\delta,\gamma) \text{ or} \\
       % & \quad \bullet\; w\cdot (\beta,\alpha) = w\cdot (\delta,\gamma) \text{ and } v_1\cdot (\beta,\alpha) > v_1 \cdot (\delta,\gamma), \\
       & \quad \bullet\; w\cdot (\beta,\alpha) = w\cdot (\delta,\gamma) \text{ and there exists an }1\leq l \leq d \text{ with} \\
       & \quad \qquad v_i\cdot (\beta,\alpha) = v_i \cdot (\delta,\gamma) \text{ for all } 1\leq i<l \text{ and } v_l\cdot (\beta,\alpha) > v_l \cdot (\delta,\gamma) \text{ or} \\
       & \quad \bullet\; w\cdot (\beta,\alpha) = w\cdot (\delta,\gamma) \text{ and } v_i\cdot (\beta,\alpha) = v_i \cdot (\delta,\gamma) \text{ for all } 1\leq i\leq d\\
       & \quad \qquad \text{and } t^\beta\cdot x^\alpha > t^\delta\cdot x^\gamma.
     \end{align*}
     Moreover, given a polyhedral cone $\sigma\subseteq\RR_{\leq 0}\times\RR^n$ of dimension $d$ with $\sigma\nsubseteq\{0\}\times\RR^n$ and a point $w\in\relint(\sigma)$ (note that $w\in\RR_{<0}\times\RR^n$ necessarily), we call a weight vector $u\in\sigma$ \emph{generic around $w$}, if for all open neighbourhoods $U$ around $w$ there exists a weight vector $u'\in U\cap\sigma$ not lying on any Gr\"obner cone of dimension lower than $d$ such that $\initial_{u'}(\pi^{-1}I)=\initial_u(\pi^{-1}I)$.
   \end{definition}

   \begin{algorithm}[$\initial_{\sigma,w}(G)$, generic initial ideal around a weight]\label{alg:genericInitialIdeal}\
     \begin{algorithmic}[1]
       \REQUIRE{$(\sigma,w,G)$, where
         \begin{enumerate}[leftmargin=*]
         \item $\sigma\subseteq\RR_{\leq 0}\times\RR^n$, a polyhedral cone with $\sigma\nsubseteq\{0\}\times\RR^n$,
         \item $w\in\relint(\sigma)$ a relative interior point,
         \item $G\subseteq I$ a generating set of an $x$-homogeneous ideal $I\unlhd \Rtx$.
         \end{enumerate}}
       \ENSURE{$(H',G',>')=\initial_{(\sigma,w)}(G)$, where
         \begin{enumerate}[leftmargin=*]
         \item $>':=>_u$ for a weight vector $u\in\sigma$ generic around $w$,
         \item $G'$ an initially reduced standard basis of $I$ with respect to $>'$,
         \item $H'=\{\initial_u(g)\mid g\in G'\}$.
         \end{enumerate}}
       \STATE Choose a basis $v_1,\ldots,v_d$ of the linear span of $\sigma$.
       \STATE Pick a $t$-local monomial ordering $>$ on $\Mon(t,x)$.
       \STATE Compute an initially reduced standard basis $G'$ of $I=\langle G \rangle$ w.r.t. $>_{(w,v_1,\ldots,v_d)}$.
       \STATE Set $H':=\{\initial_{(w,v_1,\ldots,v_d)}(g)\mid g\in G'\}$
       \RETURN{$(H',G',>_{(w,v_1,\ldots,v_d)})$.}
     \end{algorithmic}
   \end{algorithm}

   \begin{definition}\label{def:tropicalWitness}
     Let $I\unlhd\Rtx$ and let $u\in\RR_{<0}\times\RR^n$ be such that $C_u(I)\nsubseteq \Trop(I)$. We call an element $f\in I$ a \emph{tropical witness} of $C_u(I)$ if $\initial_v(f)$ is a monomial for all $v\in\Relint(C_u(I))$.
   \end{definition}

   \begin{algorithm}[TropWitness, {\cite[Algorithm 4.7]{BJSST07}}]\label{alg:witness}\
     \begin{algorithmic}[1]
       \REQUIRE{$(m,H,G,>)$, where
         \begin{enumerate}[leftmargin=*]
         \item $>_w$ a weighted monomial ordering for some $w\in\RR_{<0}\times\RR^n$,
         \item $G=\{g_1,\ldots,g_k\}$ an initially reduced standard basis of an $x$-homogeneous ideal $I\unlhd\Rtx$ with respect to $>_w$,
         \item $H=\{h_1,\ldots,h_k\}$ with $h_i=\initial_w(g_i)$,
         \item $m\in\initial_w(I)$ a monomial.
         \end{enumerate}}
       \ENSURE{$f\in I$, a tropical witness of $C_w(I)$.}
       \STATE Compute a standard representation $m = q_1\cdot h_1+\ldots+q_k\cdot h_k$, i.e. no term of $q_i\cdot h_i$ lies in $\langle \lm_>(h_j)\mid j<i\rangle$ for all $1\leq i\leq k$.
       \RETURN{$q_1\cdot g_1+\ldots+q_k\cdot g_k$.}
     \end{algorithmic}
   \end{algorithm}

   \begin{algorithm}[TropStar, {\cite[Algorithm 4.8]{BJSST07}}]\label{alg:tropCurve}\
     \begin{algorithmic}[1]
       \REQUIRE{$G$, the generating set of an $x$-homogeneous ideal $I\unlhd\Rtx$ with $\dim \Trop(I)=\dim C_0(I)+1$.}
       \ENSURE{$\Delta$, a collection of maximal dimensional polyhedral cones in $\RR_{\leq 0}\times\RR^n$ covering $\Trop(I)$.}
       \STATE Compute the common refinement of all tropical hypersurfaces, throwing away cones in $\{0\}\times\RR^n$,
       \begin{center}
         $\Delta:=\{\sigma\in \textstyle\bigwedge_{g\in G}\text{TropHypersurface}(g)\mid \sigma\cap \RR_{<0}\times\RR^n\neq\emptyset \}$.
       \end{center}
       \STATE Set $L:=\Delta$.
       \WHILE{$L\neq\emptyset$}
       \STATE Pick $\sigma\in L$ maximal and $w\in\relint(\sigma)$.
       \STATE Compute an initial ideal with respect to a weight $w\in\sigma$ generic around $w$:
       \begin{center}
         $(H',G',>')=\initial_{(\sigma,w)}(G).$
       \end{center}
       \IF{$\initial_u(I)=\langle H'\rangle$ contains a monomial $s\neq 0$}
       \STATE Compute a tropical witness $g:=\text{TropicalWitness}(s,H',G',>')$.
       \STATE Set
       \begin{center}
         $G:=G\cup\{ g \},\quad \Delta := \Delta\wedge \Trop(g),\quad L:=L\wedge\Trop(g).$
       \end{center}
       \STATE \textbf{continue}
       \ENDIF
       \STATE Suppose $w=(w_t,w_x)\in\RR_{<0}\times\RR^n$, set $w_{\text{neg}}:=(w_t,-w_x)\in\RR_{<0}\times\RR^n$.
       \IF{$w_{\text{neg}}\in\sigma$}
       \STATE Redo Steps $5$ to $9$ with $w:=w_{\text{neg}}$.
       \ENDIF
       \STATE Set $L:=L\setminus\{\sigma\}$.
       \ENDWHILE
       \RETURN{$\Delta$}
     \end{algorithmic}
   \end{algorithm}

   \begin{example}\label{ex:tropOneCodimension}
     Consider the ideal $I\unlhd\ZZ\llbracket t\rrbracket[x_1,\ldots,x_4]$ generated by
     \begin{displaymath}
       g_0:=3, \quad g_1:=tx_3x_4-tx_1x_2+x_1^2, \quad
       g_2:=tx_1x_2^2-x_1^2x_2-t^3x_1x_2x_3+t^2x_1^2x_3,
     \end{displaymath}
     which is $3$-dimensional with $C_0(I)=\Cone((-2,-1,1,5,-5))+\RR\cdot(0,1,1,1,1)$. Figure \ref{fig:TropG02} illustrates the combinatorial structure of $\Trop(g_0)\cap\Trop(g_1)\cap\Trop(g_2)$.
     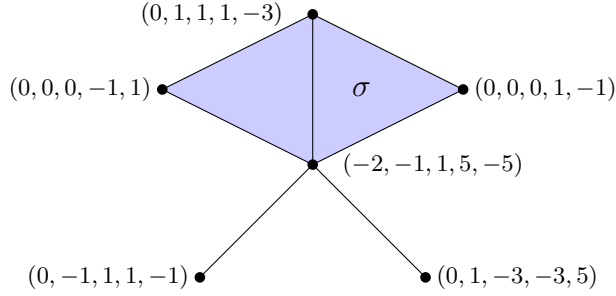
\begin{figure}[h]
       \centering
       \begin{tikzpicture}
         \fill[blue!20] (0,0) -- (0,2) -- (2,1) -- cycle;
         \fill[blue!20] (0,0) -- (0,2) -- (-2,1) -- cycle;
         \draw (0,0) -- (0,2)
         (0,0) -- (2,1)
         (0,0) -- (-2,1)
         (2,1) -- (0,2)
         (-2,1) -- (0,2)
         (0,0) -- (1.5,-1.5)
         (0,0) -- (-1.5,-1.5);
         \fill (0,0) circle (2pt)
         (0,2) circle (2pt)
         (2,1) circle (2pt)
         (-2,1) circle (2pt)
         (1.5,-1.5) circle (2pt)
         (-1.5,-1.5) circle (2pt);
         \node[anchor=west, font=\scriptsize, xshift=0.25cm] at (0,0) {$(-2,-1,1,5,-5)$};
         \node[anchor=west, font=\scriptsize] at (2,1) {$(0,0,0,1,-1)$};
         \node[anchor=west, font=\scriptsize] at (1.5,-1.5) {$(0,1,-3,-3,5)$};
         \node[anchor=east, font=\scriptsize, xshift=-0.25cm] at (0,2) {$(0,1,1,1,-3)$};
         \node[anchor=east, font=\scriptsize] at (-2,1) {$(0,0,0,-1,1)$};
         \node[anchor=east, font=\scriptsize] at (-1.5,-1.5) {$(0,-1,1,1,-1)$};
         \node at (0.65,1) {$\sigma$};
       \end{tikzpicture}\vspace{-0.5cm}%
       \caption{combinatorial structure of $Trop(g_0)\cap\Trop(g_1)\cap\Trop(g_2)$}
       \label{fig:TropG02}
     \end{figure}

     One cone $\sigma$ that can be seen to be contained in the intersection is
     \begin{displaymath}
       \sigma := \Cone(\underbrace{(0,0,0,-1,1)}_{=:w_1}, \underbrace{(0,1,1,1,-3)}_{=:w_2})+C_0(I),
     \end{displaymath}
     because we have
     \begin{displaymath}
       g_1=\underbrace{tx_3x_4\overbrace{\textcolor{red}{-tx_1x_2+x_1^2}}^{\initial_{w_2}(g_1)}}_{\initial_{w_1}(g_1)},\quad
       g_2=\underbrace{\overbrace{\textcolor{red}{tx_1 x_2^2-x_1^2x_2}}^{\initial_{w_1}(g_2)}-t^3x_1 x_2 x_3+t^2x_1 ^2x_3}_{\initial_{w_2}(g_2)},
     \end{displaymath}
     so that for any weight $w\in\sigma$, $\initial_w(g_1)$ contains
     at least the binomial $-tx_1x_2+x_1^2$ and $\initial_w(g_2)$
     contains at least the binomial $tx_1x_2-x_1^2x_2$. In particular,
     neither are monomials.
     However, it can be shown that, for
     \begin{displaymath}
       g_3:=tx_2x_3x_4+t^2x_1^2x_3-t^3x_1x_2x_3 \in I \quad \text{and} \quad w:=(-1,1,2,2,0)\in\sigma,
     \end{displaymath}
     $\initial_w(g_3)=tx_2x_3x_4$ is a monomial, which implies
     that $\sigma\nsubseteq\Trop(I)$ (not that we would have expected
     otherwise considering
     $\dim(\sigma)=4>3=\dim(\Trop(I))$). Figure~\ref{fig:TropG03}
     illustrates the combinatorial structure of $\bigcap_{i=0}^3
     \Trop(g_i)$, red highlighting all weights that have been eliminated through
     the intersection with $\Trop(g_3)$.
     \begin{figure}[h]
       \centering
       \begin{tikzpicture}
         \fill[blue!20] (0,0) -- (0,2) -- (1.33,1.33) -- cycle;
         \fill[blue!20] (0,0) -- (0,2) -- (-2,1) -- cycle;
         \fill[red!20] (0,0) -- (1.33,1.33) -- (2,1) -- cycle;
         \draw[red!70] (1.33,1.33) -- (2,1);
         \draw (0,0) -- (0,2)
         (0,0) -- (2,1)
         (0,0) -- (-2,1)
         (1.33,1.33) -- (0,2)
         (-2,1) -- (0,2)
         (0,0) -- (1.5,-1.5)
         (0,0) -- (-1.5,-1.5)
         (0,0) -- (1.33,1.33);
         \fill (0,0) circle (2pt)
         (0,2) circle (2pt)
         (2,1) circle (2pt)
         (-2,1) circle (2pt)
         (1.5,-1.5) circle (2pt)
         (-1.5,-1.5) circle (2pt)
         (1.33,1.33) circle (2pt);
         \node[anchor=west, font=\scriptsize, xshift=0.25cm] at (0,0) {$(-2,-1,1,5,-5)$};
         \node[anchor=west, font=\scriptsize] at (2,1) {$(0,0,0,1,-1)$};
         \node[anchor=west, font=\scriptsize] at (1.5,-1.5) {$(0,1,-3,-3,5)$};
         \node[anchor=base west, font=\scriptsize] at (1.33,1.33) {$(0,1,1,-3,1)$};
         \node[anchor=east, font=\scriptsize, xshift=-0.25cm] at (0,2) {$(0,1,1,1,-3)$};
         \node[anchor=east, font=\scriptsize] at (-2,1) {$(0,0,0,-1,1)$};
         \node[anchor=east, font=\scriptsize] at (-1.5,-1.5) {$(0,-1,1,1,-1)$};
         \node at (0.45,1) {$\sigma'$};
       \end{tikzpicture}\vspace{-0.5cm}%
       \caption{combinatorial structure of $\bigcap_{i=0}^3\Trop(g_i)$}
       \label{fig:TropG03}
     \end{figure}

     Continuing, the following cone can be seen to be contained in the intersection of the tropical varieties of $g_0,\ldots,g_3$,
     \begin{displaymath}
       \sigma':=\Cone(\underbrace{(0,1,1,-3,1)}_{=:w_1},\underbrace{(0,1,1,1,-3)}_{=:w_2}) + C_0(I),
     \end{displaymath}
     since
     \begin{align*}
       g_1&=tx_3x_4\overbrace{\textcolor{red}{-tx_1x_2+x_1^2}}^{\initial_{w_1}(g_1),\;\initial_{w_2}(g_1)},\quad
       g_2=\underbrace{\overbrace{\textcolor{red}{tx_1 x_2^2-x_1^2x_2}}^{\initial_{w_1}(g_2)} -t^3x_1 x_2 x_3 +t^2x_1 ^2x_3}_{\initial_{w_2}(g_2)}, \\
       g_3&=\underbrace{tx_2x_3x_4\overbrace{+\textcolor{red}{t^2x_1^2x_3-t^3x_1x_2x_3}}^{\initial_{w_3}(g_3)}}_{\initial_{w_1}(g_3)}.
     \end{align*}
     However, setting
     \begin{displaymath}
       g_4:=tx_2x_3x_4-t^3x_3^2x_4 \in I \quad \text{and} \quad  w':=(-1,3,4,5,0)\in\sigma',
     \end{displaymath}
     $\initial_{w'}(g_4)=tx_2x_3x_4$ is a monomial. Hence, we have again $\sigma'\nsubseteq\Trop(I)$ and Figure~\ref{fig:TropG04}
     illustrates the combinatorial structure of $\bigcap_{i=0}^4
     \Trop(g_i)$.
     Further calculations will yield that indeed $\Trop(I)=\bigcap_{i=0}^4\Trop(g_i)$.

     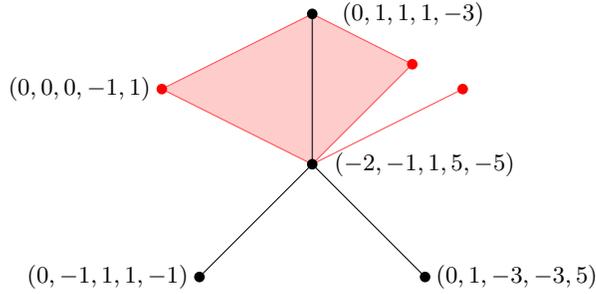
\begin{figure}[h]
       \centering
       \begin{tikzpicture}
         \fill[red!20] (0,0) -- (0,2) -- (1.33,1.33) -- cycle;
         \fill[red!20] (0,0) -- (0,2) -- (-2,1) -- cycle;
         \draw (0,0) -- (0,2)
         (0,0) -- (1.5,-1.5)
         (0,0) -- (-1.5,-1.5);
         \draw[red!70] (0,0) -- (2,1)
         (1.33,1.33) -- (0,2)
         (0,0) -- (-2,1)
         (-2,1) -- (0,2)
         (0,0) -- (1.33,1.33);
         \fill[red] (2,1) circle (2pt)
         (1.33,1.33) circle (2pt)
         (-2,1) circle (2pt);
         \fill (0,0) circle (2pt)
         (0,2) circle (2pt)
         (1.5,-1.5) circle (2pt)
         (-1.5,-1.5) circle (2pt);
         \node[anchor=west, font=\scriptsize, xshift=0.15cm] at (0,0) {$(-2,-1,1,5,-5)$};
         \node[anchor=west, font=\scriptsize] at (1.5,-1.5) {$(0,1,-3,-3,5)$};
         \node[anchor=west, font=\scriptsize, xshift=0.25cm] at (0,2) {$(0,1,1,1,-3)$};
         \node[anchor=east, font=\scriptsize] at (-2,1) {$(0,0,0,-1,1)$};
         \node[anchor=east, font=\scriptsize] at (-1.5,-1.5) {$(0,-1,1,1,-1)$};
       \end{tikzpicture}\vspace{-0.5cm}%
       \caption{combinatorial structure of $\bigcap_{i=0}^4\Trop(g_i)$}
       \label{fig:TropG04}
     \end{figure}
   \end{example}

   Tropical varieties with one-codimensional homogeneity space are important as they describe general tropical varieties locally around a codimension one cone.

   \begin{example}\label{ex:tropicalStar}
     Consider again the ideal $I\unlhd \ZZ\llbracket t\rrbracket[x_1,\ldots,x_4]$ from Example~\ref{ex:Chan} and \ref{ex:ChanStartingCone} generated by
     \[ 3-t, \quad 8tx_3x_4+tx_1x_2+2x_1^2, \quad tx_1x_2^2+2x_1^2x_2+2t^3x_1x_2x_3+4t^2x_1^2x_3-64tx_1^3. \]
     For any weight vector inside $\Trop(I)$, say $w=(-2,-1,1,5,-5)$, $\Trop(\initial_w(I))$ describes $\Trop(I)$ locally around $w$, see Figure~\ref{fig:tropStarComparison}.
     In particular, if $w$ lies on a Gr\"obner cone of codimension $1$,
     we have
     \[\dim C_0(\initial_w(I))\overset{\text{Lem.}}{\underset{\text{\ref{lem:homogeneitySpace}}}{=}}\dim C_{w}(I)=\dim \Trop(I)-1 = \dim \Trop(\initial_w(I))-1, \]
     which allows us to compute $\Trop(\initial_w(I))$ using Algorithm~\ref{alg:tropCurve}.
     \begin{figure}[h]
       \centering
       \begin{tikzpicture}
         \draw (-2,-1) -- (0,-0.5) -- (2,-1);
         \draw (2,-1) -- (3.5,0.5)
         (2,-1) -- (0,-3.5)
         (-2,-1) -- (-3.5,0.5)
         (-2,-1) -- (0,-3.5)
         (0,-0.5) -- (0,1.5);
         \node[anchor=north, font=\scriptsize, yshift=-0.05cm] at (0,-0.5) {$w$};
         \draw[very thick,blue!70,->] (0,-0.5) -- (0.5,-0.63);
         \draw[very thick,blue!70,->] (0,-0.5) -- (0,0);
         \draw[very thick,blue!70,->] (0,-0.5) -- (-0.5,-0.63);
         \fill (0,-0.5) circle (2pt)
         (-2,-1) circle (2pt)
         (2,-1) circle (2pt)
         (3.5,0.5) circle (2pt)
         (-3.5,0.5) circle (2pt)
         (0,1.5) circle (2pt)
         (0,-3.5) circle (2pt);
         \node[anchor=east, font=\scriptsize, xshift=-0.2cm] at (-2,-1) {$(-1,0,-1,1,0)$};
         \node[anchor=west, font=\scriptsize, xshift=0.2cm] at (2,-1) {$(-1,0,1,3,-4)$};
         \node[anchor=south, font=\scriptsize, xshift=0.2cm] at (3.5,0.5) {$(0,1,-3,1,1)$};
         \node[anchor=south, font=\scriptsize, xshift=-0.2cm] at (-3.5,0.5) {$(0,1,1,5,-7)$};
         \node[anchor=north, font=\scriptsize, xshift=-0.2cm] at (0,-3.5) {$(0,0,0,-1,1)$};
         \node[anchor=south, font=\scriptsize, xshift=0.2cm] at (0,1.5) {$(0,-1,1,1,-1)$};

         \draw (7,-1.5) -- (5,-2)
           (7,-1.5) -- (9,-2)
           (7,-1.5) -- (7,0.5);
         \draw[very thick,blue!70,->] (7,-1.5) -- (6.5,-1.63);
         \draw[very thick,blue!70,->] (7,-1.5) -- (7,-1);
         \draw[very thick,blue!70,->] (7,-1.5) -- (7.5,-1.63);
         \fill (7,-1.5) circle (2pt)
           (5,-2) circle (2pt)
           (9,-2) circle (2pt)
           (7,0.5) circle (2pt);
         \node[anchor=north, font=\scriptsize, yshift=-0.05cm] at (7,-1.5) {$w$};
         \node[anchor=north, font=\scriptsize] at (9,-2) {$(0,1,-3,-3,5)$};
         \node[anchor=south, font=\scriptsize, xshift=0.05cm] at (7,0.5) {$(0,1,1,1,-3)$};
         \node[anchor=north, font=\scriptsize] at (5,-2) {$(0,-1,1,1,-1)$};
       \end{tikzpicture}\vspace{-0.5cm}%
       \caption{$\Trop(I)$ and $\Trop(\initial_w(I))$}
       \label{fig:tropStarComparison}
     \end{figure}
   \end{example}

   Combining Algorithms~\ref{alg:tropStartingCone}, \ref{alg:tropCurve} and \ref{alg:flip}, we obtain an algorithm to compute the tropical variety of a general ideal, provided it is pure and connected in codimension one.

   \begin{algorithm}[Trop, {\cite[Algorithm 4.11]{BJSST07}}]\label{alg:tropVariety}\
     \begin{algorithmic}[1]
       \REQUIRE{$(G_{\text{input}},>_{\text{input}})$, where for an $x$-homogeneous ideal $I\unlhd\Rtx$ with $\Trop(I)$ pure and connected in codimension one:
         \begin{itemize}[leftmargin=*]
         \item $>_{\text{input}}$ is a weighted monomial ordering,
         \item $G_{\text{input}}$ an initially reduced standard basis of $I$ with respect to $>_{\text{input}}$.
         \end{itemize}}
       \ENSURE{$\Delta=\{C_w(I)\mid C_w(I)\in\Trop(I) \text{ maximal}\}$, so that
         \begin{center}
           $\Trop(I)=\textstyle\bigcup_{C_w(I)\in\Delta} C_w(I).$
         \end{center}}
       \STATE Compute a starting cone
       \begin{center}
         $(C_w(I),G,>) = \text{TropStartingCone}(G_{\text{input}},>_{\text{input}}).$
       \end{center}
       \STATE Initialize $\Delta:=\{C_w(I)\}$.
       \STATE Initialize a working list $L:=\{(G,>,C_w(I))\}$.
       \WHILE{$L\neq \emptyset$}
       \STATE Pick $(G,>,C_w(I))\in L$.
       \FOR{all facets $\tau\leq C_w(I)$, $\tau\nsubseteq\{0\}\times\RR^n$}
       \STATE Compute a relative interior point $u\in\tau$.
       \STATE Set $H:=\{\initial_u(g)\mid g\in G\}$.
       \STATE Compute the tropical star
       \begin{center}
         $\Delta_{\text{star}} = \text{TropStar}(H).$
       \end{center}
       \FOR{$\theta\in \Delta_{\text{star}}$}
       \STATE Compute a relative interior point $v\in\theta$.
       \IF{$C_{u+\varepsilon\cdot v}(I)\notin \Delta$ for $\varepsilon>0$ sufficiently small}
       \STATE Flip the standard basis to the adjacent ordering
       \begin{center}
         $(G',>'):=\Flip(G,H,v,>).$
       \end{center}
       \STATE Set $H':=\{\initial_{(u,v)}(g)\mid g\in G'\}$.
       \STATE Construct the adjacent Gr\"obner cone
       \begin{center}
         $C_{w'}(I):=C(H',G',>').$
       \end{center}
       \STATE Set
       \begin{center}
         $\Delta:=\Delta\cup\{C_{w'}(I)\}\quad\text{and}\quad L:=L\cup\{ (G',>',C_{w'}(I))\}.$
       \end{center}
       \ENDIF
       \ENDFOR
       \ENDFOR
       \STATE Set $L:=L\setminus\{(G,>,C_w(I))\}$.
       \ENDWHILE
       \RETURN{$\Delta$.}
     \end{algorithmic}
   \end{algorithm}
   % \begin{proof}
   %   Termination and correctness follow from Lemma \ref{lem:subfan} as well as the correctness of all the algorithms used.
   % \end{proof}

   \begin{example}[tropical traversal]
   \end{example}\vspace{-0.75cm}
     For a visual example of Algorithm~\ref{alg:tropVariety} at work, consider the $3$-dimensional ideal
     \begin{align*}
       I &= \langle 4x^2+xy+16y^2+xz+8z^2, 2-t \rangle \\
       &= \langle \underbrace{t^2x^2+xy+t^4y^2+xz+t^3z^2}_{=:g}, 2-t \rangle \in \ZZ\llbracket t \rrbracket [x,y,z].
     \end{align*}
     As $\initial_w(2-t)=2$ for all $w\in\RR_{<0}\times\RR^3$, it suffices to solely focus on $g$.
     For the starting cone, we begin with weight vector $w=(-3,-10,1,0)\in\RR_{<0}\times \RR^3$, since $\initial_w(g)=xy+t^3z^2$ is no monomial. In fact, its initial form is binomial, hence the only weight vectors $v$ such that $\initial_{w+\varepsilon v}(g)$ is no monomial are the $v$ such that $\initial_{w+\varepsilon v}(g) = \initial_{w}(g)$, or in other words $v\in C_{w}(I)$. This shows that $C_w(I)$ is a maximal cone in the tropical variety.

\begin{wrapfigure}{l}{0.3\textwidth}
\centering
       \begin{tikzpicture}
         \draw (0,0) -- (2,0);
         \fill (0,0) circle (2pt);
         \fill (2,0) circle (2pt);
         \node[anchor=south, font=\footnotesize] at (0,0) {$v_1$};
         \node[anchor=south, font=\footnotesize] at (2,0) {$v_2$};
       \end{tikzpicture}
\end{wrapfigure}
     Note that all Gr\"obner cones are invariant under translation by
     $(0,1,1,1)$. Hence the $3$-dimensional Gr\"obner cone $C_{w}(I)$ is
     spanned by two rays, which are generated by $v_1=(-2,-7, 1,0)$ and
     $v_2=(-1,-3, 0,0)$ respectively. This can be seen from their
     respective initial forms, which gain one additional term
     compared to $\initial_{w}(g)$,
     $\initial_{v_1}(g)=xy+t^4y^2+t^3z^2$ and $\initial_{v_2}(g) = xy+xz+t^3z^2$.
     We have thus finished computing a starting cone and identified its two facets, which we need to traverse.

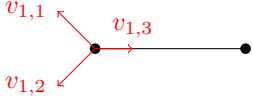
\begin{wrapfigure}{l}{0.3\textwidth}
\centering
       \begin{tikzpicture}
         \draw (0,0) -- (2,0);
         \fill (0,0) circle (2pt);
         \fill (2,0) circle (2pt);
         \draw[red,->] (0,0) -- (0.5,0);
         \draw[red,->] (0,0) -- (-0.5,0.5);
         \draw[red,->] (0,0) -- (-0.5,-0.5);
         \node[anchor=south, font=\footnotesize, red] at (0.5,0) {$v_{1,3}$};
         \node[anchor=east, font=\footnotesize, red] at (-0.5,0.5) {$v_{1,1}$};
         \node[anchor=east, font=\footnotesize, red] at (-0.5,-0.5) {$v_{1,2}$};
       \end{tikzpicture}
\end{wrapfigure}

     If we pick one of the facets, say the one generated by $v_1$, we see
     that its tropical star consists of three rays. One ray points in the
     direction $v_{1,3}=(0,0,-2,-1)$ so that
     $\initial_{v_1+\varepsilon\cdot
       v_{1,3}}(g) = xy+t^3z^2=\initial_w(g)$, which undoubtedly points into our
     starting cone. Another ray points in the direction
     $v_{1,2}=(0,0,1,1)$ so that $\initial_{v_1+\varepsilon\cdot
       v_{1,2}}(g) = t^4y^2+t^3z^2$. The last ray points in the direction
     $v_{1,1}=(0,0,0,-1)$ so that $\initial_{v_1+\varepsilon\cdot v_{1,1}}(g) =
     xy+t^4y^2$.

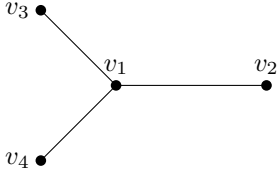
\begin{wrapfigure}{l}{0.3\textwidth}
\centering
       \begin{tikzpicture}
         \draw (0,0) -- (2,0);
         \fill (0,0) circle (2pt);
         \fill (2,0) circle (2pt);
         \fill (-1,1) circle (2pt);
         \fill (-1,-1) circle (2pt);
         \draw (0,0) -- (-1,1);
         \draw (0,0) -- (-1,-1);
         \node[anchor=south, font=\footnotesize] at (0,0) {$v_1$};
         \node[anchor=south, font=\footnotesize] at (2,0) {$v_2$};
         \node[anchor=east, font=\footnotesize] at (-1,1) {$v_3$};
         \node[anchor=east, font=\footnotesize] at (-1,-1) {$v_4$};
       \end{tikzpicture}
\end{wrapfigure}

     Continuing with direction $v_{1,2}=(0,0,1,1)$, to whose side lies the
     closure of equivalence class such that
     $\initial_{w'}(g)=t^4y^2+t^3z^2$, we see that the other ray of the
     maximal Gr\"obner cone is generated by $v_3=(0,0,1,1)$ with
     $\initial_{v_3}(g)=t^4y^2+t^3z^2$. The ray lies on the boundary of the
     maximal Gr\"obner cone because it lies on the boundary of the lower
     halfspace.

     Continuing with the direction $v_{1,1}=(0,0,0,-1)$, which is the
     closure of the equivalence class such that
     $\initial_{w'}(g)=xy+t^4y^2$, we get that the other ray of the
     maximal Gr\"obner cone is $v_4 = (0,0,0,-1)$ with
     $\initial_{v_4}(g)=t^2x^2+xy+t^4y^2$.

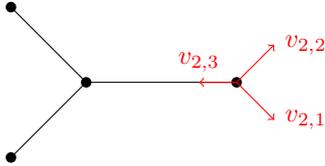
\begin{wrapfigure}{l}{0.4\textwidth}
\centering
       \begin{tikzpicture}
         \draw (0,0) -- (2,0);
         \fill (0,0) circle (2pt);
         \fill (2,0) circle (2pt);
         \fill (-1,1) circle (2pt);
         \fill (-1,-1) circle (2pt);
         \draw (0,0) -- (-1,1);
         \draw (0,0) -- (-1,-1);
         \draw[->,red] (2,0) -- (1.5,0);
         \draw[->,red] (2,0) -- (2.5,0.5);
         \draw[->,red] (2,0) -- (2.5,-0.5);
         \node[red, anchor=south, font=\footnotesize] at (1.5,0) {$v_{2,3}$};
         \node[red, anchor=west, font=\footnotesize] at (2.5,0.5) {$v_{2,2}$};
         \node[red, anchor=west, font=\footnotesize] at (2.5,-0.5) {$v_{2,1}$};
       \end{tikzpicture}
\end{wrapfigure}
     Because both $v_3$ and $v_4$ lie on the boundary of the lower
     halfspace, the only facet left to traverse is the one generated by
     $v_2$. The tropical star around $v_2$ consists of three rays. One ray
     points in the direction of $v_{2,1}=(0,1,0,0)$ so that
     $\initial_{v_2+\varepsilon\cdot v_{2,1}}(g)=xy+xz$. Another ray points
     in the direction of $v_{2,2}=(0,0,-1,0)$ so that
     $\initial_{v_2+\varepsilon\cdot v_{2,2}}(g)=xz+t^3z^2$. The final ray
     points in the direction of $v_{2,3}=(0,0,2,1)$ so that
     $\initial_{v_2+\varepsilon\cdot v_{2,3}}(g)=xy+t^3z^2=\initial_w(g)$, this is the
     vector pointing into our starting cone.

\begin{wrapfigure}{l}{0.4\textwidth}
\centering
       \begin{tikzpicture}
         \draw (0,0) -- (2,0);
         \fill (0,0) circle (2pt);
         \fill (2,0) circle (2pt);
         \fill (-1,1) circle (2pt);
         \fill (-1,-1) circle (2pt);
         \draw (0,0) -- (-1,1);
         \draw (0,0) -- (-1,-1);
         \draw (2,0) -- (3,1);
         \draw (2,0) -- (3,-1);
         \fill (3,1) circle (2pt);
         \fill (3,-1) circle (2pt);
         \node[anchor=south, font=\footnotesize] at (0,0) {$v_1$};
         \node[anchor=south, font=\footnotesize] at (2,0) {$v_2$};
         \node[anchor=east, font=\footnotesize] at (-1,1) {$v_3$};
         \node[anchor=east, font=\footnotesize] at (-1,-1) {$v_4$};
         \node[anchor=west, font=\footnotesize] at (3,1) {$v_6$};
         \node[anchor=west, font=\footnotesize] at (3,-1) {$v_5$};
       \end{tikzpicture}
\end{wrapfigure}

     Continuing in the direction of $v_{2,1}$, the other ray of the
     maximal Gr\"obner cone is generated by $v_5=(-1,2,0,0)$ as
     $\initial_{v_5}(g)=t^2x^2+xy+xz$. And continuing in the
     direction of $v_{2,2}$, the other ray is generated by
     $v_6:=(0,0,-1,0)$ as $\initial_{v_6}(g)=t^2x^2+xz+t^3z^2$.

     Because $v_6$ lies on the boundary of the lower halfspace, $v_5$
     generates the only facet left to traverse. A quick glance at the
     initial forms imply that it is connected to the facets generated by
     $v_4$ and $v_6$, as it has two terms in common with each of them.

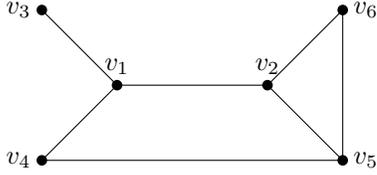
\begin{wrapfigure}{l}{0.4\textwidth}
\centering
       \begin{tikzpicture}
         \draw (0,0) -- (2,0);
         \fill (0,0) circle (2pt);
         \fill (2,0) circle (2pt);
         \fill (-1,1) circle (2pt);
         \fill (-1,-1) circle (2pt);
         \draw (0,0) -- (-1,1);
         \draw (0,0) -- (-1,-1);
         \draw (2,0) -- (3,1);
         \draw (2,0) -- (3,-1);
         \draw (3,1) -- (3,-1);
         \draw (3,-1) -- (-1,-1);
         \fill (3,1) circle (2pt);
         \fill (3,-1) circle (2pt);
         \node[anchor=south, font=\footnotesize] at (0,0) {$v_1$};
         \node[anchor=south, font=\footnotesize] at (2,0) {$v_2$};
         \node[anchor=east, font=\footnotesize] at (-1,1) {$v_3$};
         \node[anchor=east, font=\footnotesize] at (-1,-1) {$v_4$};
         \node[anchor=west, font=\footnotesize] at (3,1) {$v_6$};
         \node[anchor=west, font=\footnotesize] at (3,-1) {$v_5$};
       \end{tikzpicture}
\end{wrapfigure}

     We obtain that $\Trop(I)$ is covered by a polyhedral fans which,
     modulo the homogeneity space $\RR\cdot (0,1,1,1)$,
     has $6$ rays, of which the ones generated by $v_1,v_2,v_5$ lie in the
     interior of the lower halfspace $\RR_{\leq 0}\times\RR^n$, while the
     ones generated by $v_3,v_4,v_6$ lie on its boundary.

     The $6$ rays are pairwise connected via $7$ edges.
     The edges connecting $(v_1,v_3)$, $(v_1,v_4)$,
     $(v_2,v_6)$ and $(v_4,v_5)$ intersect the boundary in codimension
     one, while the cones connecting $(v_1,v_2)$ and $(v_2,v_5)$
     intersect the boundary in codimension $2$, which has to be the
     homogeneity space.

   \begin{example}[dependency on the valuation]\label{ex:valuationDependency}
     Consider the ideal from Example \ref{ex:chanTwoAdic},
     $I:=\langle x_1-2x_2+3x_3,3x_2-4x_3+5x_4 \rangle \unlhd \QQ[x]$.
     Figure~\ref{fig:chanVarious} shows its tropical varieties for all possible valuations on $\mathbb Q$.
     Regardless of the valuation, all tropical varieties share the same recession fan, as was proven by Gubler \cite{Gub13}. The latter is also necessarily the tropical variety under the trivial valuation.
     Note that for $p$ sufficiently large, the tropical varieties under $\nu_p$ coincides with
     the tropical variety under the trivial valuation. This is
     because $p$ is simply too large for $p-t$ to
     matter in any of our standard basis calculations.
     These $p$ are referred to as \emph{good primes} while other $p$ are referred to as \emph{bad primes} in the theory of modular techniques \cite{BDFLP2016}.
   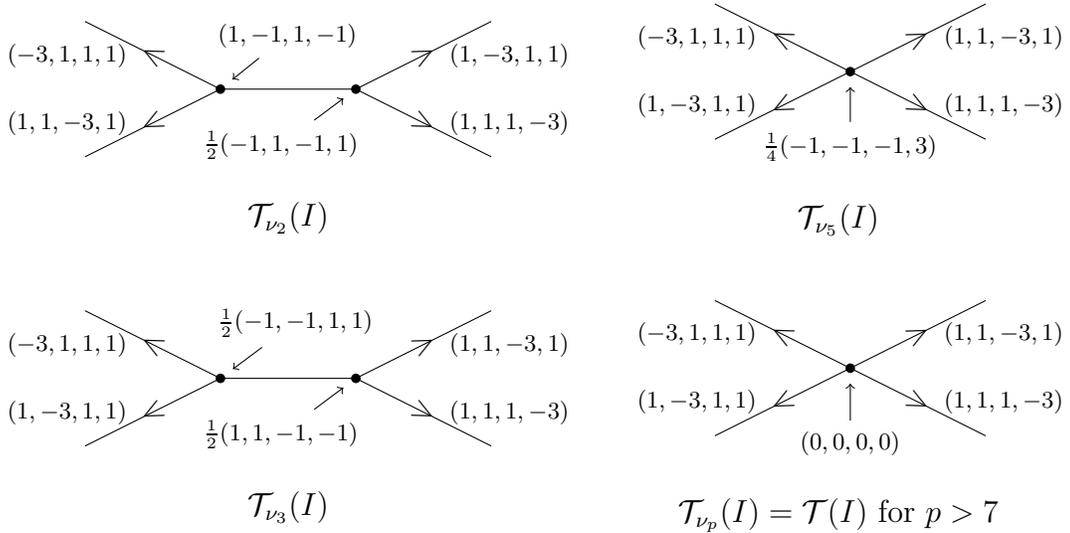
\begin{figure}[h]
     \centering
     \begin{minipage}{0.6\linewidth}
       \centering
       \begin{tikzpicture}[scale=0.9]
         \draw (0,0) -- (2,0);
         \fill (0,0) circle (2pt);
         \fill (2,0) circle (2pt);
         \draw (0,0) -- node[sloped] (a) {$<$} (-2,1)
         (0,0) -- node[sloped] (b) {$<$} (-2,-1)
         (2,0) -- node[sloped] (c) {$>$} (4,1)
         (2,0) -- node[sloped] (d) {$>$} (4,-1);
         \node[anchor=south west,font=\scriptsize,xshift=-5,yshift=1em] (v1) at (0,0) {$(1,-1,1,-1)$};
         \node[anchor=north east,font=\scriptsize,xshift=5,yshift=-1em] (v2) at (2,0) {$\frac{1}{2}(-1,1,-1,1)$};
         \node[anchor=east,font=\scriptsize, xshift=-0.2cm] at (a) {$(-3,1,1,1)$};
         \node[anchor=east,font=\scriptsize, xshift=-0.2cm] at (b) {$(1,1,-3,1)$};
         \node[anchor=west,font=\scriptsize, xshift=0.2cm] at (c) {$(1,-3,1,1)$};
         \node[anchor=west,font=\scriptsize, xshift=0.2cm] at (d) {$(1,1,1,-3)$};
         \draw[->,shorten >= 0.5em,thin] (v1) -- (0,0);
         \draw[->,shorten >= 0.5em,thin] (v2) -- (2,0);
       \end{tikzpicture}\vspace{0.5em}\\
       $\Trop_{\hspace{-0.1cm}\nu_2}(I)$\vspace{2em}

       \begin{tikzpicture}[scale=0.9]
         \draw (0,0) -- (2,0);
         \fill (0,0) circle (2pt);
         \fill (2,0) circle (2pt);
         \draw (0,0) -- node[sloped] (a) {$<$} (-2,1)
         (0,0) -- node[sloped] (b) {$<$} (-2,-1)
         (2,0) -- node[sloped] (c) {$>$} (4,1)
         (2,0) -- node[sloped] (d) {$>$} (4,-1);
         \node[anchor=south west,font=\scriptsize,xshift=-5,yshift=1em] (v1) at (0,0) {$\frac{1}{2}(-1,-1,1,1)$};
         \node[anchor=north east,font=\scriptsize,xshift=5,yshift=-1em] (v2) at (2,0) {$\frac{1}{2}(1,1,-1,-1)$};
         \node[anchor=east,font=\scriptsize, xshift=-0.2cm] at (a) {$(-3,1,1,1)$};
         \node[anchor=east,font=\scriptsize, xshift=-0.2cm] at (b) {$(1,-3,1,1)$};
         \node[anchor=west,font=\scriptsize, xshift=0.2cm] at (c) {$(1,1,-3,1)$};
         \node[anchor=west,font=\scriptsize, xshift=0.2cm] at (d) {$(1,1,1,-3)$};
         \draw[->,shorten >= 0.5em,thin] (v1) -- (0,0);
         \draw[->,shorten >= 0.5em,thin] (v2) -- (2,0);
       \end{tikzpicture}\vspace{0.5em}\\
       $\Trop_{\hspace{-0.1cm}\nu_3}(I)$
     \end{minipage}%
     \begin{minipage}{0.375\linewidth}
       \centering
       \begin{tikzpicture}[scale=0.9]
         \fill (0,0) circle (2pt);
         \node[font=\scriptsize,yshift=-1cm] (v1) at (0,0) {$\frac{1}{4}(-1,-1,-1,3)$};
         \draw (0,0) -- node[sloped] (a) {$<$} (-2,1)
         (0,0) -- node[sloped] (b) {$<$} (-2,-1)
         (0,0) -- node[sloped] (c) {$>$} (2,1)
         (0,0) -- node[sloped] (d) {$>$} (2,-1);
         \node[anchor=east,font=\scriptsize, xshift=-0.2cm] at (a) {$(-3,1,1,1)$};
         \node[anchor=east,font=\scriptsize, xshift=-0.2cm] at (b) {$(1,-3,1,1)$};
         \node[anchor=west,font=\scriptsize, xshift=0.2cm] at (c) {$(1,1,-3,1)$};
         \node[anchor=west,font=\scriptsize, xshift=0.2cm] at (d) {$(1,1,1,-3)$};
         \draw[->,shorten >= 0.5em,thin] (v1) -- (0,0);
       \end{tikzpicture}\vspace{0.5em}\\
       $\Trop_{\hspace{-0.1cm}\nu_5}(I)$\vspace{2em}

       \begin{tikzpicture}[scale=0.9]
         \fill (0,0) circle (2pt);
         \node[font=\scriptsize,yshift=-1cm] (v1) at (0,0) {$(0,0,0,0)$};
         \draw (0,0) -- node[sloped] (a) {$<$} (-2,1)
         (0,0) -- node[sloped] (b) {$<$} (-2,-1)
         (0,0) -- node[sloped] (c) {$>$} (2,1)
         (0,0) -- node[sloped] (d) {$>$} (2,-1);
         \node[anchor=east,font=\scriptsize, xshift=-0.2cm] at (a) {$(-3,1,1,1)$};
         \node[anchor=east,font=\scriptsize, xshift=-0.2cm] at (b) {$(1,-3,1,1)$};
         \node[anchor=west,font=\scriptsize, xshift=0.2cm] at (c) {$(1,1,-3,1)$};
         \node[anchor=west,font=\scriptsize, xshift=0.2cm] at (d) {$(1,1,1,-3)$};
         \draw[->,shorten >= 0.5em,thin] (v1) -- (0,0);
       \end{tikzpicture}\vspace{0.5em}\\
       $\Trop_{\hspace{-0.1cm}\nu_p}(I)=\Trop(I)$ for $p>7$
     \end{minipage}\vspace{-0.25cm}%
     \caption{$\Trop_{\hspace{-0.1cm}\nu}(I)$ for various $p$-adic and the trivial valuations.}
     \label{fig:chanVarious}
   \end{figure}
   \end{example}

   % %%%%%%%%%%%%%%%%%%%%%%%%%%%%%%%%%%%%%%%%%%%%%%%%%%%%%%%%%%%
   \begin{example}[independency of the valuation, \textsc{Singular} output]
     Consider the following ideal of Grassmann-Pl\"ucker relations for $\Grass(2,5)$,
     \begin{align*}
       I\;:=\;\langle& x_1x_5-x_0x_7-x_2x_4, x_1x_6-x_0x_8-x_3x_4, x_2x_6-x_0x_9-x_3x_5, \\
       & x_2x_8-x_1x_9-x_3x_7, x_5x_8-x_4x_9-x_6x_7 \rangle \unlhd \;\QQ[x_0,\ldots,x_9].
     \end{align*}
     Unlike Example~\ref{ex:valuationDependency}, its tropical variety does not seem to dependent on the choice of valuation, which is not surprising as Speyer and Sturmfels showed that it is characteristic-free \cite[Theorem 7.1]{SpeyerSturmfels04}. In this case, the computations under the $p$-adic valuation are mathematically equivalent to the computations under the trivial valuation, though the practical timings under the $p$-adic valuation are slightly slower due to a constant overhead of a more general framework.

     Figure~\ref{fig:plueckerOutput} shows a shortened output of
     \textsc{Singular} when computing its tropical variety with respect
     to the $2$-adic valuation. It describes a polyhedral fan whose
     intersection with the affine hyperplane $\{-1\}\times\RR^{10}$ yields again a polyhedral fan:
     The ray \#0 represents the 5-dimensional lineality space of $\Trop_{\hspace{-0.1cm}\nu_2}(I)$,
     while the maximal cones \texttt{\{0\;i\;j\}} represent polyhedral cones in $\Trop_{\hspace{-0.1cm}\nu_2}(I)$
     spanned by the lineality space and rays \#i, \#j.
     Note that, from a perspective of $\RR^n=\{-1\}\times\RR^n$, all data is given in homogenized coordinates,
     which is why the f-Vector shown is slightly distorted by lower-dimensional cones at infinity.
     \begin{figure}[p]
       \centering
       \begin{lstlisting}[language=bash,basicstyle=\ttfamily\scriptsize,columns=flexible,escapeinside={<@}{@>}]
         SINGULAR                                                       /
         A Computer Algebra System for Polynomial Computations         /   Version 4.1.0
                                                                     0<
         by: W. Decker, G.-M. Greuel, G. Pfister, H. Schoenemann       \   Dec 2016
         FB Mathematik der Universitaet, D-67653 Kaiserslautern         \
         > LIB "gfanlib.so";
         > printlevel = 1;
         > ring r=0,(a,b,c,d,e,f,g,h,i,j),dp;
         > ideal I = bf-ah-ce, bg-ai-de, cg-aj-df, ci-bj-dh, fi-ej-gh;
         > tropicalVariety(I,number(2));
         cones finished: 1   cones in working list: 4
         [...] <@\textcolor{red}{information on the state of the traversal because printlevel=1 was set}@>
         cones finished: 14   cones in working list: 1
         cones finished: 15   cones in working list: 0
       \end{lstlisting}\vspace{-0.5cm}
       \begin{lstlisting}[language=bash,basicstyle=\ttfamily\scriptsize,columns=flexible,multicols=2]
         _application PolyhedralFan
         _version 2.2
         _type PolyhedralFan

         AMBIENT_DIM
         11

         DIM
         8

         LINEALITY_DIM
         5

         RAYS
         -1 0 0 0 0 0 0 0 0 0 0   # 0
         0 -3 1 1 1 1 1 1 -1 -1 -1# 1
         0 -1 -1 1 1 -1 1 1 1 1 -3# 2
         0 -1 1 -1 1 1 -1 1 1 -3 1# 3
         0 -1 1 1 -1 1 1 -1 -3 1 1# 4
         0 1 -3 1 1 1 -1 -1 1 1 -1# 5
         0 1 -1 -1 1 1 1 -3 -1 1 1# 6
         0 1 -1 1 -1 1 -3 1 1 -1 1# 7
         0 1 1 -3 1 -1 1 -1 1 -1 1# 8
         0 1 1 -1 -1 -3 1 1 1 1 -1# 9
         0 1 1 1 -3 -1 -1 1 -1 1 1# 10

         LINEALITY_SPACE
         0 -1 0 0 0 0 0 0 1 1 1     # 0
         0 0 -1 0 0 0 1 1 0 0 1     # 1
         0 0 0 1 0 0 1 0 1 0 1      # 2
         0 0 0 0 1 0 0 1 0 1 1      # 3
         0 0 0 0 0 -1 -1 -1 -1 -1 -1# 4

         F_VECTOR
         1 11 25 15

         MAXIMAL_CONES
         {0 1 2}# Dimension 8
         {0 1 3}
         {0 1 4}
         {0 2 5}
         {0 2 9}
         {0 3 7}
         {0 4 6}
         {0 3 8}
         {0 4 10}
         {0 5 6}
         {0 5 7}
         {0 6 8}
         {0 7 10}
         {0 8 9}
         {0 9 10}
       \end{lstlisting}
       \caption{\textsc{Singular} output for the Grassmann-Pl\"ucker ideal}
       \label{fig:plueckerOutput}
     \end{figure}

     Figure \ref{fig:plueckerIllustration} illustrates the combinatorial
     structure of $\Delta$. Each vertex represents a ray of $\Delta$,
     while each edge represents a maximal cone of $\Delta$.
     The graph shown should be thought of as lying on a sphere $S^2$, on
     which the colored edges connect with their counterpart on the other
     side.
     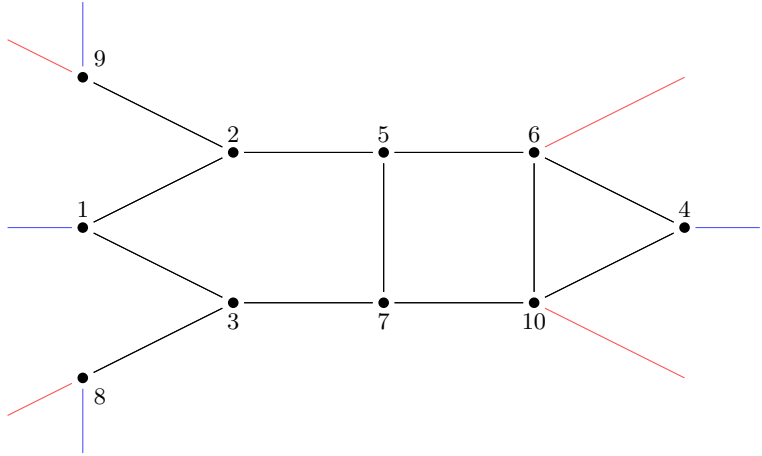
\begin{figure}[h]
       \centering
       \begin{tikzpicture}
         \node (n1) at (0,0) {};
         \node (n2) at (2,1) {};
         \node (n3) at (2,-1) {};
         \node (n5) at (4,1) {};
         \node (n7) at (4,-1) {};
         \node (n6) at (6,1) {};
         \node (n10) at (6,-1) {};
         \node (n4) at (8,0) {};
         \node (n8) at (0,-2) {};
         \node (n9) at (0,2) {};

         \fill (n1) circle (2pt);
         \fill (n2) circle (2pt);
         \fill (n3) circle (2pt);
         \fill (n4) circle (2pt);
         \fill (n5) circle (2pt);
         \fill (n6) circle (2pt);
         \fill (n7) circle (2pt);
         \fill (n8) circle (2pt);
         \fill (n9) circle (2pt);
         \fill (n10) circle (2pt);

         \draw (n1) -- (n2)
         (n1) -- (n3)
         (n2) -- (n5)
         (n2) -- (n9)
         (n3) -- (n7)
         (n3) -- (n8)
         (n5) -- (n6)
         (n5) -- (n7)
         (n7) -- (n10)
         (n6) -- (n10)
         (n6) -- (n4)
         (n10) -- (n4);
         \draw (n1) -- (n2)
         (n1) -- (n3)
         (n2) -- (n5)
         (n2) -- (n9)
         (n3) -- (n7)
         (n3) -- (n8)
         (n5) -- (n6)
         (n5) -- (n7)
         (n7) -- (n10)
         (n6) -- (n10)
         (n6) -- (n4)
         (n10) -- (n4);
         \draw [blue!70]
         (n1) -- (-1,0)
         (n4) -- (9,0)
         (n9) -- ($(n9)+(0,1)$)
         (n8) -- ($(n8)+(0,-1)$);
         \draw [red!70]
         (n6) -- ($(n6)+(2,1)$)
         (n8) -- ($(n8)+(-1,-0.5)$)
         (n10) -- ($(n10)+(2,-1)$)
         (n9) -- ($(n9)+(-1,0.5)$);

         \node [anchor=south, font=\scriptsize] at (n1) {$1$};
         \node [anchor=north, font=\scriptsize] at (n3) {$3$};
         \node [anchor=north, font=\scriptsize] at (n7) {$7$};
         \node [anchor=north, font=\scriptsize] at (n10) {$10$};
         \node [anchor=south, font=\scriptsize] at (n2) {$2$};
         \node [anchor=south, font=\scriptsize] at (n5) {$5$};
         \node [anchor=south, font=\scriptsize] at (n6) {$6$};
         \node [anchor=south, font=\scriptsize] at (n4) {$4$};
         \node [anchor=south west, font=\scriptsize] at (n9) {$9$};
         \node [anchor=north west, font=\scriptsize] at (n8) {$8$};
       \end{tikzpicture}\vspace{-0.5cm}%
       \caption{tropical variety of the Grassmann-Pl\"ucker ideal}
       \label{fig:plueckerIllustration}
     \end{figure}
   \end{example}

   %%%%%%%%%%%%%%%%%%%%%%%%%%%%%%%%%%%%%%%%%%%%%%%%%%%%%%%%%%%
   \section{Optimizations for non-trivial valuations}\label{sec:optimization}

   Up till now, all algorithms for computing $\Trop_{\hspace{-0.1cm}\nu}(I)$ via $\Trop(\pi^{-1}I)$
   appear to be strictly worse than computing $\Trop(I)$, as
   we are working with an inhomogeneous ideal $\pi^{-1}I$ over a coefficient ring $R$
   instead of a homogeneous ideal $I$ over a coefficient field $K$.
   In this section, however, we consider simple optimizations for the traversal,
   which suggest that working under a nontrivial valuation need not necessarily be slower than working under a trivial valuation.

   The main algebraic bottlenecks in the computation of tropical varieties are:
   \begin{enumerate}
   \item computing generic initial ideals, Algorithm~\ref{alg:genericInitialIdeal} ,
   \item checking generic initial ideals for monomials in Algorithm~\ref{alg:tropCurve},
   \item the flip of standard bases, Algorithm~\ref{alg:flip},
   \end{enumerate}
   all of which require at least one standard basis computation, which
   is the reason for the bottleneck. However, from
   Algorithm~\ref{alg:tropVariety}, they are never called on the
   actual input ideal, they are exclusively called on its initial
   ideals instead. This can be exploited, should the input ideal of
   Algorithm~\ref{alg:tropVariety} be of the form $\pi^{-1}I$ for some
   $I\unlhd K[x]$. Lemma~\ref{lem:StantardBases} then shows that many
   computations can actally be done over the residue field $\mathfrak{K}$.

   \begin{convention}
     Let $I\unlhd K[x]$ be a homogeneous ideal and fix an initial ideal $J:=\initial_{(-1,w)}(\pi^{-1}I)\unlhd R[t,x]$ of its preimage as well as the corresponding monomial ordering $>_{(-1,w)}$. Note that necessarily $p\in J$.
   \end{convention}

   \begin{lemma}[quasi-homogeneity of $J$]\label{lem:optOrdering}
     There exists a positive weight vector $u\in(\RR_{>0})^{n+1}$ such that $J$ is weighted homogeneous with respect to it.
   \end{lemma}
   \begin{proof}
     Because $J$ is weighted homogeneous with respect to $w\in\RR_{<0}\times\RR^n$ and $x$-homogeneous, it is, picking $k\in\NN$ sufficiently high, also weighted homogeneous with respect to
     $k\cdot(0,1,\ldots,1)+w\in (\RR_{>0})^{n+1}$.
   \end{proof}

   \begin{definition}
     We call an element $g=\sum_{\beta,\alpha} c_{\beta,\alpha}\cdot t^\beta x^\alpha \in R[t,x]$ a \emph{canonical representative} of its residue class $\overline{g}\in \mathfrak K[t,x]$, if
     \[ c_{\beta,\alpha}=0\;\; \Longleftrightarrow\;\; \overline{c}_{\beta,\alpha}=\overline 0\qquad \text{and} \qquad
        c_{\beta,\alpha}=1\;\; \Longleftrightarrow\;\; \overline{c}_{\beta,\alpha}=\overline 1. \]
   \end{definition}

   \begin{lemma}[standard bases of $J$]\label{lem:StantardBases}
     Let $\{\overline g_1,\ldots,\overline g_k\}$ be a monic standard basis of $\overline J$ with respect to $>_{(-1,w)}$. Then $\{g_0,g_1,\ldots,g_k\}$ is a standard basis of $J$ with respect to $>_{(-1,w)}$, where $g_0=p$ and $g_1,\ldots,g_k$ are canonical representatives of their residue classes.
   \end{lemma}
   \begin{proof}
     Let $G=\{g_0,\ldots,g_k\}$. Since $p\in J$, it is clear that $\langle G \rangle\subseteq J$ and therefore
     $\langle \lt_>(g) \mid g \in G\rangle\subseteq\lt_>(J)$.
     For the converse, consider a term $s=c\cdot t^\beta x^\alpha\in\lt_>(J)$.
     Now if $p \mid c$, then $s\in\langle \lt_>(g) \mid g \in G \rangle$,
     since $p\in G$ and $\lt_>(p)=p$.
     And if $p\nmid c$, we may use $p\in\lt_>(J)$ to normalize $s$, and get $t^\beta
     x^\alpha\in\lt_>(J)$. Thus $t^\beta x^\alpha\in\lt_>(\ov{J})$, and
     hence there is a $\overline{g_i}$ such that $\lm_>(\overline
     g_i)\mid t^\beta x^\alpha$. Since all $\ov g_i$ were chosen to be
     monic, this implies $\lt_>(\ov g_i)\mid t^\beta x^\alpha$, and
     because all $g_i$ were chosen to be canonical representatives, this
     implies $\lt_>(g_i)\mid s$.
   \end{proof}

   This article was dedicated to show how $\Trop_{\hspace{-0.1cm}\nu}(I)$ can be computed via $\Trop(\pi^{-1}I)$,
   however until now we have not addressed how to compute the preimage $\pi^{-1}I$ in the first place.
   We will therefore end the article with two results: The first will show that $\pi^{-1}I$ can be obtained by a saturation.
   The second will allow us get around computing the saturation.

   \begin{lemma}\label{lem:preimage}
     Let $I \unlhd K[x]$ be an ideal, and let
     $G=\{g_1,\ldots,g_k\}\subseteq I\cap \mathcal O_K[x]$ be a generating set
     over the valuation ring. Since $\pi:\Rtx\rightarrow \mathcal O_K[x]$ is
     surjective, there exist $g_1',\ldots,g_k'\in\Rtx$ such that
     $\pi(g_i')=g_i\in R[x]$. Then
     \begin{displaymath}
       \pi^{-1}I = \Big(\langle g_1',\ldots,g_k' \rangle + \langle p-t \rangle \Big): p^\infty  \unlhd \Rtx.
     \end{displaymath}
   \end{lemma}
   \begin{proof}
     $\pi^{-1}I \supseteq (\langle g_1',\ldots,g_k' \rangle + \langle p-t
     \rangle) : p^\infty $ is obvious, as $p-t$ is mapped to $0$ and $p$
     is invertible in $K$.

     For the converse inclusion, let $f\in\pi^{-1}I$. Then there are $q_1,\ldots,q_k\in K[x]$ such that
     \begin{displaymath}
       \pi(f)=q_1\cdot g_1 + \ldots + q_k\cdot g_k \in K[x],
     \end{displaymath}
     which means that for a sufficiently high power $l\in \NN$ we have
     \begin{displaymath}
       p^l\cdot \pi(f)=\underbrace{p^lq_1}_{\in \mathcal O_K[x]}\cdot g_1 + \ldots + \underbrace{p^lq_k}_{\in \mathcal O_K[x]}\cdot g_k \in \mathcal O_K[x].
     \end{displaymath}
     Since the map $\pi:\Rtx\rightarrow \mathcal O_K[x]$ is surjective, there exist $q_1',\ldots,q_k'\in \Rtx$ such that
     \begin{displaymath}
       p^l\cdot \pi(f)=\pi(q_1' \cdot g_1' + \ldots + q_k' \cdot g_k'),
     \end{displaymath}
     or rather
     \begin{displaymath}
       p^l\cdot f - q_1' \cdot g_1' + \ldots + q_k' \cdot g_k'\in \ker(\pi)=\langle p-t\rangle.
     \end{displaymath}
     Thus $p^l\cdot f \in \langle g_1',\ldots,g_k' \rangle + \langle p-t \rangle$, and hence
     \begin{displaymath}
       f\in (\langle g_1',\ldots,g_k' \rangle + \langle p-t \rangle): p^\infty. \qedhere
     \end{displaymath}
   \end{proof}

   \begin{proposition}\label{prop:tropPreimage}
     Let $I \unlhd K[x]$ be an ideal, and let
     $G=\{g_1',\ldots,g_k'\}\subseteq \pi^{-1}I$ such that $I=\langle
     \pi(g_1'),\ldots,\pi(g_k')\rangle$. Then
     \begin{displaymath}
       \Trop(\pi^{-1}I) = \Trop(\langle g_1',\ldots,g_k'\rangle + \langle p-t \rangle).
     \end{displaymath}
   \end{proposition}
   \begin{proof}
     By Lemma \ref{lem:preimage}, we have
     \begin{displaymath}
       \pi^{-1}I = \Big(\underbrace{\langle g_1',\ldots,g_k' \rangle + \langle p-t \rangle}_{=:I'} \Big): p^\infty  \unlhd \Rtx.
     \end{displaymath}
     Consider a weight vector $w\in\RR_{<0}\times\RR^n$ and suppose
     $\initial_w(I')$ contains a monomial $t^\beta x^\alpha$. By
     Algorithm \ref{alg:witness}, there exists a witness $f\in I'$ with
     $\initial_w(f)=t^\beta x^\alpha$. However since
     $I'\subseteq\pi^{-1}I$, $\initial_w(\pi^{-1}I)$ then contains the monomial
     $t^\beta x^\alpha$ as well.

     Now suppose $\initial_w(\pi^{-1}I)$ contains a monomial $t^\beta
     x^\alpha$. By Algorithm \ref{alg:witness}, there exists a witness
     $f\in\pi^{-1}I$ with $\initial_w(f)=t^\beta x^\alpha$. Let $l\in\NN$
     be sufficiently high such that $p^l\cdot f\in I'$. Now since $p-t\in
     I'$, this implies $t^l\cdot f\in I'$ and $\initial_w(I')$ then
     contains the monomial $\initial_w(t^l\cdot f)=t^{\beta+l}
     x^\alpha$.
   \end{proof}

  \renewcommand{\emph}[1]{\textit{#1}}
%   \bibliographystyle{amsalpha-tom}
%   \bibliography{tropicalvarieties}

\newcommand{\etalchar}[1]{$^{#1}$}
\providecommand{\bysame}{\leavevmode\hbox to3em{\hrulefill}\thinspace}

\end{document}